\documentclass[12pt]{amsart}
\usepackage{amssymb}

\newcommand{\eps}{\varepsilon}

\newcommand{\N}{{\mathbb N}}
\newcommand{\C}{{\mathbb C}}

\newcommand{\R}{{\mathbb R}}

\newcommand{\wand}{wandering domain}
\newcommand{\mcwd}{multiply connected wandering domain}

\newcommand{\tef}{transcendental entire function}
\newcommand{\tmf}{transcendental meromorphic function}

\newcommand{\ef}{entire function}

\newcommand{\sconn}{simply connected}
\newcommand{\mconn}{multiply connected}
\newcommand{\dconn}{doubly connected}

\theoremstyle{plain}
\newtheorem{theorem}{Theorem}[section]

\newtheorem*{thma}{Theorem~A}
\newtheorem*{thmb}{Theorem~B}

\newtheorem{lemma}[theorem]{Lemma}
\theoremstyle{definition}

\theoremstyle{remark}

\theoremstyle{problem}

\theoremstyle{example}
\newtheorem{example}{Example}

\begin{document}


\title[Multiply connected wandering domains of entire functions]{Multiply connected wandering domains of entire functions}

\author{Walter Bergweiler}
\address{Mathematisches Seminar\\Christian-Albrechts-Universit\"at zu Kiel\\Ludewig-Meyn-Str. 4\\ D-24098 Kiel\\Germany}
\email{bergweiler@math.uni-kiel.de}

\author{Philip J. Rippon}
\address{Department of Mathematics and Statistics \\
   Walton Hall\\
   The Open University \\
   Milton Keynes MK7 6AA\\
   UK}
\email{p.j.rippon@open.ac.uk}

\author{Gwyneth M. Stallard}
\address{Department of Mathematics and Statistics \\
   Walton Hall\\
   The Open University \\
   Milton Keynes MK7 6AA\\
   UK}
\email{g.m.stallard@open.ac.uk}

\thanks{The first author is supported by Deutsche Forschungsgemeinschaft, Be 1508/7-2, and the ESF Networking Programme HCAA, and last two authors are supported by EPSRC grant EP/H006591/1.}





\begin{abstract}
The dynamical behaviour of a {\tef} in any periodic component of the Fatou set is well understood. Here we study the dynamical behaviour of a {\tef} $f$ in any multiply connected wandering domain~$U$ of $f$. By introducing a certain positive harmonic function $h$ in $U$, related to harmonic measure, we are able to give the first detailed description of this dynamical behaviour. Using this new technique, we show that, for sufficiently large $n$, the image domains $U_n=f^n(U)$ contain large annuli, $C_n$, and that the union of these annuli acts as an absorbing set for the iterates of $f$ in $U$. Moreover, $f$ behaves like a monomial within each of these annuli and the orbits of points in $U$ settle in the long term at particular `levels' within the annuli, determined by the function $h$. We also discuss the proximity of $\partial U_n$ and $\partial C_n$ for large $n$, and the connectivity properties of the components of $U_n \setminus \overline{C_n}$. These properties are deduced from new results about the behaviour of an entire function which omits certain values in an annulus.

\end{abstract}

\maketitle

\section{Introduction}
\setcounter{equation}{0} The \!{\it Fatou set} $F(f)$ of a transcendental entire (or rational) function $f$ is the subset of the plane $\C$ (or the Riemann sphere $\widehat{\C}=\C\cup\{\infty\}$) where the iterates~$f^n$ of $f$ form a normal family. The complement of $F(f)$ in $\C$ (or $\widehat{\C}$ for rational $f$) is called the {\it Julia set} $J(f)$ of $f$. An introduction to the properties of these sets for rational functions can be found in~\cite{aB91},~\cite{CG},~\cite{jM06} or~\cite{nS93}, and for {\ef}s in~\cite{wB93}.

The set $F(f)$ is completely invariant under $f$ in the sense
that $z\in F(f)$ if and only if $f(z)\in F(f)$. Therefore, if $U$ is a component of $F(f)$, a so-called {\it Fatou component}, then there
exists, for each $n=0,1,2,\ldots\,$, a Fatou component $U_n$ such that $f^n(U) \subset U_n$.  If, for some
$p\ge 1$, we have $U_p =U_0= U$, then we say that $U$ is a {\em
periodic} component of {\it period} $p$,
assuming $p$ to be minimal. If $U_n$ is not eventually periodic, then $U$ is a {\it wandering domain} of $f$.

A famous theorem of Sullivan \cite{dS82} says that rational functions do not have wandering domains. As the
dynamical behaviour in periodic components  is well understood (see \cite{aB91}, \cite{CG}, \cite{jM06} or \cite{nS93}) this leads to a complete description of the dynamics of rational functions on their Fatou sets. A key point in Sullivan's argument is that a simply connected {\wand} gives rise to an infinite dimensional Teichm\"uller space associated to the function; see \cite{cMdS98} and \cite{nFcH09} for the Teichm\"uller space of rational and entire functions respectively. For {\mcwd}s a different argument has to be made; see \cite[Section~2]{dS82}, \cite[Section~3]{lB87}, \cite[Lemma~6.10]{cMdS98} and \cite{BKL4}, the last of these now being the standard argument. In fact a wandering annulus only gives rise to a one-dimensional Teichm\"uller space \cite[Theorem~6.1]{cMdS98}.

Eremenko and Lyubich \cite{EL92} as well as Goldberg and Keen \cite{lGlK} extended
Sullivan's theorem to the Speiser class ${\mathcal S}$ of entire functions
for which the set of critical and asymptotic values is finite. In order to extend Sullivan's technique they had to show that Fatou components of functions in ${\mathcal S}$ are {\sconn}. It turns out \cite[Proposition~3]{EL92} that this is in fact the case for the wider Eremenko-Lyubich class ${\mathcal B}$ consisting of all functions for which the set of critical values and asymptotic values is bounded.

Much work on the dynamics of entire functions has been concerned
with the Speiser class and the Eremenko-Lyubich class,
an important tool being a logarithmic change of variable which
yields good distortion estimates; see \cite{BKZ09}, \cite{EL92}, \cite{lR09}, \cite{RRRS} and \cite{gS96} for some results concerning the Eremenko-Lyubich class.

Although Fatou components of functions in class ${\mathcal B}$ are {\sconn}, there are many examples of entire functions with {\mconn} Fatou components. The first such function was given by Baker~\cite{iB63}, who proved later~\cite{iB76} that this function has a multiply connected Fatou
component that is a wandering domain, thus giving the first example of an entire function with a {\wand}. Moreover, Baker showed~\cite{iB75}
that this is not a special property of this example: if $U$ is any
multiply connected Fatou component of a {\tef} $f$, then $U$ is
wandering domain, and has the following properties:
\begin{itemize}
\item[(a)] each $U_n$ is bounded and multiply connected,
\item[(b)] there exists $N\in\N$ such that $U_n$ and~$0$ lie in a bounded complementary component of $U_{n+1}$, for $n\ge N$,
\item[(c)] dist\,$(U_n,0) \to \infty$ as $n\to\infty$.
\end{itemize}
We note that $U_n=f^n(U)$, for $n\in\N$, since $U$ is bounded and thus $f^n: U \to U_n$ is proper; see \cite[Corollary~1]{mH98}, \cite{aB99} or the discussion in Section~3 below. Other examples of {\tef}s having {\mconn} {\wand}s with various additional properties were given in \cite{iB85}, \cite{wB08}, \cite{BZ}, \cite{aH94} and \cite{KS06}, the last of which will be discussed later in this paper.

A number of further results are known about multiply connected wandering domains, such as \cite{RS05} that they are contained in the fast escaping set $A(f)$ introduced in \cite{BH99} which has turned out to be very important in the dynamics of entire
functions; see \cite{lR06}, \cite{RS09a} and~\cite{RRRS}, for example.

In this paper we introduce a new technique for studying {\mcwd}s of {\ef}s; we show that for each {\mcwd} $U$ there is a positive harmonic function ~$h$ defined in $U$ which gives us, for the first time, precise information about the limiting behaviour under iteration of points in $U$. In addition, this technique enables us to strengthen a number of previous geometric results about {\mcwd}s, many of which we show to be best possible.

\begin{theorem}\label{main1a}
Let $f$ be a {\tef} with a {\mconn} {\wand} $U$ and let $z_0 \in U$. Then
\begin{equation}\label{hdef}
 h(z) = \lim_{n\to \infty} \frac{\log |f^n(z)|}{\log |f^n(z_0)|}
\end{equation}
defines a non-constant positive harmonic function $h: U \to \R$, with $h(z_0)=1$.
\end{theorem}

{\it Remark}\;\;The harmonic function $h$ defined in \eqref{hdef} depends on $z_0$, but if we replace $z_0\in U$ in the definition by another point $z'_0\in U$, then the resulting harmonic limit function is just $h$ scaled by the positive factor $1/h(z'_0)$.

The function $h$ has an interesting potential theoretic relationship to the domain $U$, which we
describe in Theorem~\ref{h-measure}. Before this, we state some
properties of the image components $U_n=f^n(U)$ that can be obtained
from Theorem~\ref{main1a}. It is known that the $U_n$ are large in the sense described by Theorem~A below, a special case of a result of Zheng~\cite{jhZ06}. We use the following notation:
\[
A(r,R)=\{z:r<|z|<R\}\;\text{ and }\;\overline{A}(r,R)=\{z:r\le |z|\le R\},\;\;\text{for }0<r<R.
\]
\begin{thma}\label{Zheng}
Let $f$ be a {\tef} with
a {\mconn} {\wand} $U$. If $A \subset U$ is a domain that contains a closed curve that is not null-homotopic in $U$, then, for sufficiently large $n \in\N$,
\[
U_n \supset f^n(A)\supset A(r_n,R_n), \]
where $R_n/r_n\to\infty$ as
$n\to\infty$.
\end{thma}

It follows from Theorem~A that the Julia set of an {\ef}~$f$ with a {\mconn} {\wand} is not uniformly perfect; see \cite{aH94} and \cite{jhZ02}. Recently \cite{BZ}, criteria have been given to determine whether the boundary of a {\mconn} {\wand} of an {\ef} is uniformly perfect.

Using Theorem~\ref{main1a}, we can show that the images
of a {\mcwd} must contain annuli that are much larger than those given
by Theorem~A. Moreover, the images of {\it any} open set in such a
domain must eventually contain such large annuli.

\begin{theorem}\label{main1}
Let $f$ be a {\tef} with a {\mcwd} $U$. For each $z_0 \in U$ and each open set $D \subset U$ containing $z_0$,
 there exists $\alpha > 0$ such that, for sufficiently large $n \in \N$,
    \[
     U_n \supset f^n(D) \supset A(|f^n(z_0)|^{1-\alpha}, |f^n(z_0)|^{1+\alpha}).
    \]
\end{theorem}

It follows from Theorem~\ref{main1} that, for large $n \in \N$,
 \[
 U_n \supset A(r_n, R_n),\; \mbox{ with }\; \liminf_{n\to \infty} \frac{\log R_n}{\log r_n} >1.
 \]
All previously known examples of {\mconn} {\wand}s of entire functions have the stronger property that $\log R_n / \log r_n \to \infty$ as $n \to \infty$. Later (see Example~1 and Example~3 in Section~\ref{examples}) we show that this is not true in general and that $\limsup_{n\to\infty}\log R_n/ \log r_n$ may be arbitrarily close to~$1$.

Theorem~\ref{main1a} also enables us to give a detailed description of the dynamical behaviour in any {\mcwd}, which is analogous to classical results on the behaviour in attracting and parabolic Fatou components, due to Fatou~\cite{pF19} and others (see \cite{jM06} for example), and to results describing the behaviour in escaping periodic components (Baker domains); see \cite{cC81}, \cite{nFcH06} and \cite{hK99}. That is, for any {\mcwd} of $f$ we can identify a simple `absorbing set' for $f$ (that is, a set in which the iterates of all points in $U$ eventually lie), within which the dynamical behaviour of $f$ is well-behaved.

It follows from Theorem~\ref{main1} that if $U$ is a {\mconn} Fatou component and $z_0\in U$, then there exists $\alpha > 0$ such that, for large~$n$, the {\it maximal} annulus centred at $0$ that is contained in $U_n$ and contains $f^n(z_0)$ is of the form
\begin{equation}\label{Bndef}
B_n = A(r_n^{a_n}, r_n^{b_n}),\;\;\text{where } r_n=|f^n(z_0)|, \;0<a_n<1 - \alpha < 1 + \alpha <b_n.
\end{equation}
The annuli $B_n$ have many significant properties. We begin by showing that the union of these annuli acts as an absorbing set for $f$. In fact, we can specify an absorbing set for $f$ consisting of somewhat smaller annuli.
\begin{theorem}\label{main2a}
Let $f$ be a {\tef} with a {\mconn} {\wand} $U$ and let $z_0 \in U$. Then, for each compact subset~$C$ of~$U$, there exists
$N\in\N$ such that
\[
f^n(C)\subset B_n,\;\;\text{for }n\ge N.
\]
Moreover, we have
\[
f^n(C)\subset C_n,\;\;\text{for }n\ge N,
\]
where
\begin{equation}\label{Cn}
 C_n = A\left(r_n^{a_n+2\pi\delta_n},r_n^{b_n(1-3\pi\delta_n)}\right), \;\;\mbox{with } \delta_n = 1/\sqrt{\log r_n}.
\end{equation}
\end{theorem}

{\it Remark}\;\;If the point $z_0\in U$ in the definition of $h$ is replaced by $z'_0\in U$, then it follows from Theorem~\ref{main2a} that the corresponding maximal annuli
\[
B'_n=A(|f^n(z'_0)|^{a'_n},|f^n(z'_0)|^{b'_n})\subset U_n, \;\;\text{where }0<a'_n<1<b'_n,
\]
satisfy, for large $n\in\N$,
\begin{equation}\label{z0dash}
B'_n=B_n;\;\;\text{in particular,}\;\;\frac{a'_n}{a_n}=\frac{\log|f^n(z_0)|}{\log|f^n(z'_0)|}=\frac{b'_n}{b_n}.
\end{equation}

The definition of the annuli $C_n$ appears more unnatural than that of the maximal annuli $B_n$. However, the $C_n$ are annuli on which the iterates of~$f$ behave in a particularly nice way, and the somewhat awkward exponents in the definition of $C_n$ arise from the proof of this good behaviour. For example, in Theorem~\ref{main2b} we obtain an estimate which implies that, for $m \in \N$,
\begin{equation}\label{min}
  \log m(|z|,f^m) \sim \log M(|z|,f^m),\;\;\mbox{for } z \in C_n, \mbox{ as } n \to \infty,
\end{equation}
where
\[
  m(r,f) = \min_{|z| = r}|f(z)|\; \mbox{ and } \;M(r,f) = \max_{|z| = r}|f(z)|
\]
are the {\it minimum modulus} and {\it maximum modulus} of $f$, respectively. In particular, $\limsup_{r\to\infty} \log m(r,f)/\log M(r,f)=1$. This strengthens earlier results in \cite[Proposition~2]{EL92}, \cite[Theorem~10]{wB93} and~\cite{jhZ06}.

Moreover, we can show that in $C_n$ the iterates $f^m$ behave like large degree monomials, in the sense that, for large~$n$,
\begin{equation}\label{mon}
 f^m(z) = q_{n,m} \phi_{n,m}(z)^{d_{n,m}},\;\;\mbox{for } z \in C_n,
\end{equation}
where $\phi_{n,m}$ is a conformal mapping such that $|\phi_{n,m}(z)|$ is relatively
close to $|z|$ in $C_n$, $q_{n,m}>0$ and $d_{n,m} \in \N$; see
Theorem~\ref{poly} for a precise statement. An important
consequence of Theorem~\ref{poly} is the following result.

\begin{theorem}\label{nocrits}
Let $f$ be a {\tef} with a {\mconn} {\wand} $U$, let $z_0 \in U$
and, for large $n \in \N$, let  $C_n$ be defined as
in~\eqref{Cn}. There exists $N\in\N$ such that, for $n \geq N$
and $m \in \N$, there are no critical points of $f^m$ in $C_n$.
\end{theorem}

It is natural to investigate how the sequences of annuli $(B_n)$
and $(C_n)$ relate to each other. For example, we can show that,
for $m \in \N$ and large $n \in \N$,
\begin{equation}\label{cov}
f^m(B_n) \supset A(r_{n+m}^{a_n + \delta_n}, r_{n+m}^{b_n - \delta_n})
\end{equation}
and
\begin{equation}\label{con}
f^m(C_n) \subset B_{n+m};
\end{equation}
see Theorem~\ref{cover}. We remark that $f^m(B_n)$ need not
contain $C_{n+m}$ in general -- it can be shown that  this fails
for one of the examples in Section~\ref{examples} (Example~\ref{exa2}).

The above results enable us to obtain further geometric properties of
the Fatou components $U_n$. For example, we show that
\eqref{cov} implies that the sequences of exponents $(a_n)$ and
$(b_n)$ in the definition of the annuli $B_n$ are both
convergent, including the possibility that $b_n\to\infty$ as
$n\to\infty$. Also, we can prove that, for large $n$, the annulus
$B_n$ forms a relatively large subset of $U_n$ in a certain
sense. To make this precise we define the {\it outer boundary
component} of a bounded domain~$U$ to be the boundary of the
unbounded component of $\C\setminus U$, denoted by
$\partial_{\infty}U$, and the {\it inner boundary component} of
$U$ to be the boundary of the component of $\C\setminus U$ that
contains~$0$ if there is one, denoted by $\partial_{0}U$.

\begin{theorem}\label{main2c}
Let $f$ and $U$ be as in Theorem~\ref{main1a} and, for large $n \in \N$, let $r_n$, $a_n$ and $b_n$ be defined as in \eqref{Bndef}.
\begin{itemize}
\item[(a)] Then, as $n\to\infty$,
\[
a_n \to a\in [0,1) \;\mbox{ and }\; b_n \to b\in (1,\infty].
\]
\item[(b)]
For large $n \in \N$, let $\underline{a}_{\,n}$ and $\underline{b}_{\,n}$ denote the smallest values such that
\[
\{z: |z| = r_n^{\underline{a}_{\,n}}\}\; \mbox{ and }\; \{z: |z| = r_n^{\underline{b}_{\,n}}\}
\]
 meet $\partial_{0}U$ and $\partial_{\infty}U$ respectively, and let $\overline{a}_n$ and $\overline{b}_{n}$ denote the largest such values. Then
 \begin{itemize}
 \item[(i)]
$b \geq \underline{b}_{\,n}\ge b_n$ and $b/a \geq \underline{b}_{\,n}/ \underline{a}_{\,n}\ge b_n/a_n$,  for  sufficiently large $n$;
\item[(ii)]
$\underline{b}_{\,n} \to b$, $\underline{a}_{\,n} \to a$ and $\overline{a}_n \to a$  as  $n \to \infty$.
\end{itemize}
\end{itemize}
\end{theorem}
{\it Remark}\;\; If the point $z_0\in U$ in the definition of $h$ is replaced by $z'_0\in U$, then by \eqref{z0dash} the corresponding limits~$a'$ and~$b'$ satisfy
\[
a'=a/h(z'_0)\;\;\text{and}\;\;b'=b/h(z'_0),
\]
so $b'/a'=b/a$. Also, if $b$ is finite, then so is $b'$ and if $a$ is positive, then so is $a'$.

It follows from Theorem~\ref{main2c} part~(a) and our remarks
after Theorem~\ref{main1} that all previous examples have the
property that the limits $a$ and $b$ defined in
Theorem~\ref{main2c} satisfy $b/a = \infty$. In fact it appears
that, for all these examples, $b<\infty$ and $a = 0$. In the
final section of the paper we construct several new functions
with multiply connected {\wand}s and show that there are examples
(Example~\ref{exa1} and Example~\ref{exa3}) for which $b/a <
\infty$ and other examples (Example~\ref{exa2}) for which $b =
\infty$.

One of these new examples (Example~\ref{exa3}) shows that it is
not possible to strengthen Theorem~\ref{main2c} part~(b)(ii) to
say that $\overline{b}_n \to b$. Indeed, in this example we have
$b < \infty$ and $\overline{b}_{n}/b_n \to \infty$ as $n \to
\infty$, which shows that the outer boundary component of a
{\mcwd} can be highly distorted.

Our next theorem shows how the function~$h$ introduced in Theorem~\ref{main1a} behaves near the boundary components of~$U$. Here, for any hyperbolic domain $U$ and any Borel set $E\subset \partial U$, we denote by $\omega(\cdot,E,U)$ the {\it harmonic measure} of $E$ in~$U$; for details of harmonic measure and regular boundary points, see \cite{tR95}, for example.

\begin{theorem}\label{h-measure}
Let $f$, $U$ and $h$ be as in Theorem~\ref{main1a}, and let $a$ and $b$ be defined as in Theorem~\ref{main2c}. Then
\begin{itemize}
\item[(a)]
$a=\inf_{z\in U}h(z)$ and $b=\sup_{z\in U}h(z)$;
\item[(b)]
the function $h$ extends continuously to $\partial U\setminus \partial_{\infty}U$ with
\[h(\zeta)=a,\;\;\text{for } \zeta\in\partial U\setminus \partial_{\infty}U;\]
\item[(c)] the domain $U$ is regular for the Dirichlet problem;
\item[(d)]
\begin{itemize}
\item[(i)] if $b<\infty$, then
\[h(z)=a+(b-a)\omega(z,\partial_{\infty}U,U),\;\;\text{for }z\in U;
\]
\item[(ii)] if $b=\infty$, then
\[
\omega(z,\partial_{\infty}U,U)=0,\;\;\text{for }z\in U.
\]
\end{itemize}
\end{itemize}
\end{theorem}

{\it Remark}\;\; Theorem~\ref{h-measure} part~(c) implies that any Fatou component of a {\tef} is regular for the Dirichlet problem, since this is true for {\sconn} domains. For rational functions, this property was proved by Hinkkanen~\cite{aH93}.

Theorem~\ref{h-measure} can be related to Theorem~\ref{main1a} as follows. Theorem~\ref{main1a} states, roughly speaking, that for any point $z\in U$ and large~$n$ the quantity $|f^n(z)|$ is approximately equal to $|f^n(z_0)|^{\,h(z)}$. Thus $h(z)$ can be interpreted as the `level' at which the orbit of $z$ eventually settles down, in relation to the orbit of $z_0$. It is natural to expect that the level associated with a point $z\in U$ will be highest near the outer boundary component of $U$; Theorem~\ref{h-measure} makes this idea precise.

The level sets of the function $h$ therefore consist of those points whose iterates tend to infinity `at the same rate'. In Theorem~\ref{curves} we obtain detailed results about the location of the images of these level sets. In particular, we show that, for large
~$n$, every annulus $B_n$ contains a family of simple closed
`level curves' surrounding~0, each of which maps onto another
such curve under any iterate of~$f$.

Our results also have applications to the {\it connectivity} properties of {\mconn} {\wand}s, enabling us to classify all the possibilities that can occur and relate these to the location of the critical points of $f$. In particular, Theorem~\ref{h-measure} allows us to deduce that a certain connectivity property of $U$ implies that the number $b$ introduced in Theorem~\ref{main2c} must be finite.

Let $c(G)$ denote the connectivity of a domain~$G$, that is, the number of connected components of the complement of $G$ with respect to the Riemann sphere. Kisaka and Shishikura \cite{KS06}
defined the {\it eventual connectivity} of a component~$U$ of
$F(f)$ to be $c$ provided that $c(U_n) = c$ for all
large values of~$n$ and showed that the eventual connectivity of a multiply connected wandering domain
$U$ of an entire function is either~$2$ or~$\infty$. More precisely, they used the Riemann-Hurwitz formula (see
Section~\ref{omit}) to prove the following result \cite[Theorem A]{KS06}.

\begin{thmb}\label{KS}
Let $f$ be a {\tef} with a {\mconn} {\wand} $U$ and let $n \in \N$. Then
\begin{itemize}
\item[(a)] if $c(U_n) = 2$, then $\bigcup_{m=n}^{\infty}U_m$ contains no critical points of $f$;
\item[(b)] if $c(U_n) < \infty$, then $(c(U_m))_{m\geq n}$ is a non-increasing sequence and $c(U_m) = 2$ for sufficiently large $m \in \N$;
\item[(c)] if $c(U_n) = \infty$, then $c(U_m) = \infty$ for $m>n$.
\end{itemize}
\end{thmb}

By using quasiconformal surgery, Kisaka and Shishikura \cite[Theorem B]{KS06} constructed an
example of an entire function $f$ with a {\mconn} {\wand}
with eventual connectivity~$2$, thus answering an old open question \cite[page 167]{wB93}. Earlier \cite{iB85} Baker constructed
an example which has infinite (eventual) connectivity. It has recently been shown~\cite{BZ} that Baker's original example~\cite{iB76} of a multiply connected {\wand} also has infinite connectivity.

In our next result we use Theorem~\ref{main2a} to show that in Theorem~B part~(a) the condition that $\bigcup_{m=n}^{\infty}U_m$ contains no critical points of $f$ is not only necessary but is also sufficient for $U_n$ to be doubly connected.

\begin{theorem}\label{main3}
Let $f$ be a {\tef} with a {\mconn} {\wand} $U$ and let $n \in \N$. Then
\begin{itemize}
\item[(a)] $c(U_n) = 2$ if and only if
\;$\bigcup_{m=n}^{\infty}U_m$ contains no critical points of $f$;
\item[(b)] $2 < c(U_n) < \infty$ if and only if
\;$\bigcup_{m=n}^{\infty}U_m$ contains a finite non-zero number of critical points of $f$;
\item[(c)] $U_n$ is infinitely connected if and only if
\;$\bigcup_{m=n}^{\infty}U_m$ contains infinitely many critical
points of $f$. \end{itemize}
\end{theorem}

We also characterise those {\mconn} {\wand}s $U$ which are infinitely connected but for which
the infinite connectivity occurs either near the outer boundary component
of~$U$ or near the inner boundary component of~$U$. Examples of multiply connected {\wand}s with these connectivity properties were given in~\cite{BZ}.

If $U$ is a multiply connected Fatou component, $z_0 \in U$ and $r_n = |f^n(z_0)|$, then, for large $n \in \N$, we put
\begin{equation}\label{oi}
U_n^{+} = U_n \cap \{z: |z|> r_n\}\quad\text{and}\quad  U_n^{-}= U_n \cap \{z: |z|< r_n\}.
\end{equation}
 We define $c(U_n^-)$ to be the {\it inner connectivity} of $U_n$ and $c(U_n^+)$ to be the {\it outer connectivity} of $U_n$.  We show that the results of Theorem~B and Theorem~\ref{main3} also hold for the inner and outer connectivities in most cases; see Theorem~\ref{main4a} for details. There is one interesting exception -- if the inner connectivity is infinity then the outer connectivity can only take the values $2$ and $\infty$. Further, in this case, it is possible that, for some $n \in \N$, we have $c(U_n^+) = \infty$ and $c(U_{n+1}^+) = 2$.

We say that a {\mconn} {\wand} $U$ has {\it eventual outer connectivity}~$c$ if
$U_n$ has outer connectivity $c$ for all large values of $n$.
It follows from Theorem~\ref{main4a} that $U$ has eventual outer connectivity
either~$2$ or~$\infty$. The {\it eventual inner connectivity} of
$U$ is defined similarly using the sets $U_n^{-}$ and is also equal to either~$2$ or $\infty$.

The next result, which follows from Theorem~\ref{h-measure}, gives an unexpected connection between the eventual outer connectivity of $U$ and the size of the images $U_n$.

\begin{theorem}\label{main5}
Let $f$ be a {\tef} with a {\mconn} {\wand} $U$. If~$U$ has
eventual outer connectivity~$2$, then the limit~$b$ defined in
Theorem~\ref{main2c} is finite.
\end{theorem}

Our results give detailed information about the structure of {\mcwd}s, and as one application we show that the property of having a {\mcwd} is stable under relatively small changes to the function.

\begin{theorem}\label{stable}
Let $f,g$ be transcendental entire functions and suppose that there exists $\alpha \in (0,1)$ such that
\[
 M(r,g-f) \leq  M(r,f)^{\alpha}, \;\;\mbox{for large } r.
\]
Then $g$ has a multiply connected {\wand} whenever $f$ does.
\end{theorem}
In particular, if $f$ has a multiply connected {\wand}, then so does $f+P$, for any polynomial $P$, and with a similar argument we can deduce that $f+R$, for any rational function $R$, has a {\mcwd} that satisfies properties~(a),~(b) and~(c) given before Theorem~\ref{main1a}. This answers a question of Zheng \cite[page~224]{jhZ10}.

We remark that many of the results of this paper extend, with minor changes, to {\tmf}s with a direct tract; see~\cite{BRS08} for the definition and properties of a function with a direct tract. As a first step, we note that a version of Theorem~A was proved by Zheng \cite{jhZ06} under the assumption that $f$ is a {\tmf} with finitely many poles and the proof of that result can be extended to the case of a {\tmf} with a direct tract. We do not pursue these generalisations here.

The structure of the paper is as follows. In Sections 2 and 3 we prove new results of a general nature which we need later in the paper but are also of independent interest. Section~2 includes a new convexity property of the maximum modulus of an entire function. Section~3 concerns entire functions defined on a large annulus whose images omit certain values. We show that if the image omits the unit disc, then the function behaves like a monomial inside the annulus and, under the weaker assumption that zero is omitted, we obtain strong covering results by using the hyperbolic metric.

In Section~4 we prove Theorems~\ref{main1a} and~\ref{main1}, and
in Section~5 we prove further results about the dynamical
properties of entire functions in multiply connected wandering
domains, in particular Theorem~\ref{main2a} and also
Theorems~\ref{main2b},~\ref{poly} and~\ref{cover}, from
which~\eqref{min},~\eqref{mon},~\eqref{cov} and~\eqref{con}
follow. In Section~6 we consider the geometric properties of
multiply connected wandering domains and prove
Theorem~\ref{main2c}.

Section~7 concerns the positive harmonic function~$h$; in this section we prove Theorem~\ref{h-measure} and Theorem~\ref{curves}, which describes properties of the level sets of $h$. In Section~8 we prove Theorem~\ref{main3} and Theorem~\ref{main4a}, which describes properties of the eventual outer and inner connectivity of {\mcwd}s, and also deduce Theorem~\ref{main5} from Theorem~\ref{h-measure}. In Section~9 we prove Theorem~\ref{stable} concerning the stability of {\mcwd}s. Finally, in Section~10, we construct new examples of entire functions with multiply connected {\wand}s. These contrast with earlier examples of such functions and show that several of our results cannot be strengthened.

\section{Properties of the maximum modulus function}
\setcounter{equation}{0}
In this section we give results related to the maximum modulus function. Recall that this is defined by
\( M(r,f) = \max_{|z|=r}|f(z)|,\)
which we sometimes shorten to $M(r)$ when there is no possibility of ambiguity.

To prove several of our theorems, we need the following result which is based on a technique due to
Eremenko~\cite{E} using Wiman-Valiron theory. We remark that~\cite{E} gives a special case of this
result but Eremenko's technique allows for the flexibility described
below; see \cite{RS09a} for more details and also \cite{BRS08} for a version of Wiman-Valiron theory which allows this technique to be applied to meromorphic functions with a direct tract. Here $M^n(r,f)$ denotes iteration of $M(r,f)$ with respect
to the first variable.

\begin{lemma}\label{Ere}
Let $f$ be a transcendental entire function and suppose that
$\eps>0$. There exists $R=R(f)> 0$ such that, if $r > R$, then there
exists
\[z'\in \{z : r\le |z| \le r(1 + \eps)\}\]
with
\[|f^n(z')| \ge M^n(r,f),\;\;\text{for}\; n\in \N,\]
and hence
\[
 M((1+\eps)r,f^n) \geq M^n(r,f),\;\; \text{for}\; n\in \N.
\]
\end{lemma}
Points $z'$ constructed using Eremenko's technique are particularly useful
and are sometimes referred to as {\it Eremenko points}.

Our proofs depend heavily on the following property of the maximum modulus function which is of independent interest; see~\cite[Lemma~2.2]{RS08} for the case $n=1$.
\begin{theorem}\label{convex}
Let $f$ be a {\tef}. Then there exists $R=R(f)>0$ such that, for all $0 < c' < 1 < c$ and all $n \in \N$,
\[
 \log M(r^c,f^n) \geq c\log M(r,f^n),\;\;\mbox{for } r \geq R,
\]
and
\[
\log M(r^{c'},f^n) \leq c'\log M(r,f^n),\;\;\mbox{for } r \geq R^{1/c'}.
\]
\end{theorem}
\begin{proof}
First put $\phi(t)=\log M(e^t)$, $t\in \R$. Then $\phi(t)/t\to\infty$ as $t\to\infty$, since $f$ is transcendental, so we can take $t_0$ and $t_1$, with $t_1\ge t_0>0$, so large that
\begin{equation}\label{phi1}
\phi(t_0)>0\quad\text{and}\quad\frac{\phi(t)}{t}\ge \frac{\phi(t_0)}{t_0},\;\;\text{for } t\ge t_1.
\end{equation}
Let $\phi'$ denote the right derivative of $\phi$. Then
\begin{equation}\label{phi2}
\phi'(t)\ge \frac{\phi(t)}{t},\;\;\text{for } t\ge t_1,
\end{equation}
since, by the convexity of $\phi$ and \eqref{phi1},
\[
\phi'(t)\ge\frac{\phi(t)-\phi(t_0)}{t-t_0}\ge \frac{\phi(t)}{t},\;\;\text{for } t\ge t_1.
\]
The inequality~\eqref{phi2} implies that
\begin{equation}\label{phi3}
\frac{\phi(t)}{t}\;\text{is increasing}\;\text{for } t\ge t_1,
\end{equation}
and hence that
\[
\phi(ct)\ge c\phi(t), \;\;\text{for }t\ge t_1\text{ and } c>1.
\]
This inequality is the case $n=1$ of the required result.

Now put $\phi_n(t)=\log M(e^t,f^n)$, for $t\in \R$ and $n\in \N$.
Then $\phi_1=\phi$ and for all $n\in\N$ the function $\phi_n$ is
convex. To prove the inequality in the first statement of the lemma we need to find $T$ so large that
\[
\phi_n(ct)\ge c\phi_n(t), \;\;\text{for }t\ge T,\; c>1 \text{ and } n\in\N.
\]
By arguing as we did with $\phi$ above, it is sufficient to find
$t_0$ and $T$, with $T\ge t_0$, so large that, for all $n\in\N$,
\begin{equation}\label{phi4}
\phi_n(t_0)>0\quad\text{and}\quad\frac{\phi_n(t)}{t}\ge \frac{\phi_n(t_0)}{t_0},\;\;\text{for } t\ge T.
\end{equation}
To do this we first choose $t_0>0$ so that
\begin{equation}\label{phi5}
\frac{\phi(t)}{t}>2\quad\text{and}\quad \frac{\phi(t)}{t}\text{ is increasing},\;\;\text{for }t\ge t_0/2,
\end{equation}
which is possible by \eqref{phi3}, and
\begin{equation}\label{phi6}
\phi_n(t)\ge \phi^n(t/2), \;\;\text{for }t\ge t_0,
\end{equation}
which is possible by Lemma~\ref{Ere}. The first part of
\eqref{phi4} follows immediately from \eqref{phi6} and the first part of
\eqref{phi5}. Then we use the fact that
$\phi(t)/t\to\infty$ as $t\to\infty$ to choose $T\ge 2t_0$ such
that
\begin{equation}\label{phi7}
\frac{\phi(t)}{t}\ge 2\frac{\phi(t_0)}{t_0},\;\;\text{for } t\ge T/2.
\end{equation}
Then, for $t\ge T$, we have $\phi^n(t/2)\ge \phi(t/2)\ge t\ge T\ge t_0$ for $n\in\N$ by the first part of \eqref{phi5}, so we obtain, by \eqref{phi6}, the second part of \eqref{phi5} and \eqref{phi7},
\begin{align}
\frac{\phi_n(t)}{t}&\ge\frac12\,\frac{\phi^n(t/2)}{t/2}\notag\\
&=\frac12\,\frac{\phi(\phi^{n-1}(t/2))}{\phi^{n-1}(t/2)}\,\cdots\,\frac{\phi(\phi(t/2))}{\phi(t/2)}\,\frac{\phi(t/2)}{t/2}\notag\\
&\ge \frac12\,\frac{\phi(\phi^{n-1}(t_0))}{\phi^{n-1}(t_0)}\,\cdots\,\frac{\phi(\phi(t_0))}{\phi(t_0)}\,2\frac{\phi(t_0)}{t_0}\notag\\
&=\frac{\phi^n(t_0)}{t_0}\ge \frac{\phi_n(t_0)}{t_0}\notag,
\end{align}
which gives the second part of \eqref{phi4}. This proves the first inequality in the statement of the lemma and the second inequality follows immediately.
\end{proof}
{\it Remark} \;We often use the following special case of Theorem~\ref{convex}:
\begin{equation}\label{2r}
\log M(2r,f^n)\ge (1+\log 2/\log r)\log M(r,f^n),\;\;\text{for }r\ge R, n\in \N.
\end{equation}
%

\section{Properties of functions which omit values in an annulus}\label{omit}
\setcounter{equation}{0}

In this section we prove a number of theorems concerning the behaviour of an entire function which omits certain values in an annulus. The motivation for proving these results originated with questions concerning multiply connected wandering domains, but the results should have much wider applications.

Our first aim is to show that if the image of a large annulus under an entire function omits the unit disc, then the function behaves like a monomial inside that annulus. A key step in proving this is the following result which is related to~\cite[Lemma 5 part~(a)]{RS09} and also \cite[Lemma 2]{aH94a}. We remark that Theorem~\ref{Har} is the only result from this section needed for the proofs of Theorems~\ref{main1a},~\ref{main1} and~\ref{main2a}.

\begin{theorem}\label{Har}
Let $g$ be analytic in $A(r^a,r^b)$, for some $r>1$,
$0<a<b$, and let
\begin{equation}\label{mbig1}
m(\rho,g)> 1,\;\;\text{for } \rho\in (r^a,r^b).
\end{equation}
\begin{itemize}
\item[(a)]
If $\eps \in (\pi/\log r,(b-a)/2)$, then
 \[
 \log m(\rho,g)\ge \left(1-\frac{2\pi}{\eps \log r}\right)\log M(\rho,g)>0,
 \;\;\text{for } \rho\in [r^{a+\eps},r^{b-\eps}].
 \]
\item[(b)]  In particular, if $\eps = 2\pi\delta$, where $\delta = 1/\sqrt{\log r}< \min\{2,(b-a)/(4\pi)\}$, then
  \[
 \log m(\rho,g)\ge \left(1-\delta\right)\log M(\rho,g)>0,
 \;\;\text{for } \rho\in [r^{a+2\pi\delta},r^{b-2\pi\delta}].
 \]
 \end{itemize}
\end{theorem}
\begin{proof}
By~\eqref{mbig1}, the function $u(z)=\log|g(z)|$ is positive harmonic in $A(r^a,r^b)$, so $U(t)=u(e^t)$ is positive harmonic in the strip
\[
S=\{t:a\log r <\Re(t)<b\log r\}.
\]
Now, since $\eps < (b-a)/2$, we have $(a + \eps)\log r < (b - \eps)\log r$. Thus if $t_1$ and $t_2$ satisfy $(a + \eps) \log r \le\Re(t_1)=\Re(t_2)\le (b - \eps)\log r$ and $|\Im(t_1)-\Im(t_2)|\le \pi$, then $\{t: |t-t_1| < \eps\log r\}\subset S$ and $|t_2-t_1|\le \pi< \eps \log r$. So
\[
\frac{\eps \log r -\pi}{\eps \log r +\pi}\le
\frac{U(t_2)}{U(t_1)}\le\frac{\eps \log r +\pi}{\eps \log r
-\pi}\,,
\]
by Harnack's inequality; see~\cite[Theorem~1.3.1]{tR95}. Therefore,
if $z_1$ and $z_2$ satisfy $r^{a+\eps}\le|z_1|=|z_2|\le r^{b-\eps}$,
then
\[
\frac{1-\pi/(\eps \log r)}{1+\pi/(\eps \log r)}\le
\frac{u(z_2)}{u(z_1 )}\le\frac{1+\pi/(\eps \log r)}{1-\pi/(\eps
\log r)}\,.
\]
Part (a) follows from the left-hand inequality. Part (b) follows from part~(a).
\end{proof}

We now show that, if the image of a large annulus under an entire function omits the unit disc, then the function behaves like a monomial inside the annulus. First recall \cite[Section~1.2]{nS93} that a continuous map $f:U\to V$, where $U$ and $V$ are domains, is called \emph{proper} if for each compact subset $C$ of $V$ the inverse image $f^{-1}(C)$ is also compact. We use this concept at several points in the paper. For a proper analytic function $f:U\to V$ there exists a positive integer $d$, called the \emph{degree} of $f$,
such that each point in $V$ has exactly $d$ preimages in $U$, counted according to multiplicity.
In particular, $f(U)=V$.

The Riemann-Hurwitz formula \cite[Section~1.3]{nS93} states that for a proper analytic function $f:U\to V$ of degree $d$, we have \begin{equation}\label{RieHur}
c(U)-2 = d (c(V)-2)+N,
\end{equation}
where $N$ is the number of critical points of $f$ in $U$, counted according to multiplicity. Here it is understood that $c(U)=\infty$ if and only if
$c(V)=\infty$; this case can be deduced from the case when $c(U)$ is finite by exhausting $U$ by finitely connected domains of increasing connectivity.

We note that if an analytic function $f:U\to V$ extends
continuously to the closure of $U$ (with respect to the
Riemann sphere), then $f$ is proper if and only if
$f(\partial U)=\partial V$. This implies that if $U$ is a bounded Fatou component
of an entire function $f$, then $f^n:U\to f^n(U)$ is proper, for $n \in \N$.
In particular this applies to multiply connected {\wand}s of entire functions.

\begin{theorem}\label{pol}
Let $f$ be a {\tef} such that, for some $n\in\N$, $r>1$ and $0<a<1<b$,
\begin{equation}\label{mbig}
m(\rho,f^n)> 1,\;\;\text{for } \rho\in (r^a,r^b).
\end{equation}
and suppose that $\delta=1/\sqrt{\log r}< \min\{1,(b-a)/(4\pi+5\pi b)\}$.

There exists $R=R(f)>1$ such that if $r^a \ge R$, then the preimage under $f^n$ of $A(M(r^a,f^n),M(r^{(b-2\pi\delta)(1-\delta)},f^n))$ has a {\dconn} component $A_n$ such that
\[
A(r^{a+2\pi\delta},r^{b(1-3\pi\delta)})\subset A_n\subset A(r^a,r^b)
\]
and $f^n$ has no critical points in $A_n$. Moreover, for $r^a\ge R$, we can write $f^n(z)=P_n(\phi_n(z))$, $z\in A_n$, where
\begin{itemize}
\item[(a)]
$\phi_n$ is a conformal map on $A_n$ which satisfies
\begin{equation}\label{modulus}
|z|^{1-\delta-2\pi\delta\log r/\log |z|}\le |\phi_n(z)|\le |z|^{1+2\delta},\;\;\text{for } r^{a+4\pi\delta}\le |z|\le r^{b(1-5\pi\delta)};
\end{equation}
\item[(b)]
$P_n(z)=q_nz^{d_n}$, where $q_n>0$ and the degree $d_n$ is the number of zeros of~$f^n$ in $\{z:|z|<r\}$, which satisfies
\[
 (1-2\delta)\frac{\log M(r,f^n)}{\log r}\le d_n\le \frac{\log M(r,f^n)}{(1-a)\log r}.
 \]
\end{itemize}
\end{theorem}

\begin{proof}
First note that $g=f^n$ and $\delta=1/\sqrt{\log r}$ satisfy the hypotheses of Theorem~\ref{Har} part~(b).

Consider the annulus
\begin{equation}\label{C}
C = A\left(r^{a+2\pi\delta},r^{b(1-3\pi\delta)}\right).
\end{equation}
We claim that if $R=R(f)>1$ is a constant suitable for Theorem~\ref{convex}, then for $r^a\ge R$ there is a doubly connected set $A_n\supset C$ such that $f^n$ is a proper map of $A_n$ onto the annulus
\[
 A'_n = A(M(r^a,f^n),M(r^{(b-2\pi\delta)(1-\delta)}),f^n).
\]
It follows from the Riemann-Hurwitz formula that $f^n$ has no critical points in~$A_n$.

To prove the claim, we first show that, for $r^a\ge R$, we have $f^n(C) \subset A'_n$. This is true because if $z \in C$, then
\[
|f^n(z)|\le M(r^{b(1-3\pi\delta)},f^n)< M(r^{(b-2\pi\delta)(1-\delta)},f^n),
\]
since $b>1$, and
\[
|f^n(z)|\ge M(r^{a+2\pi\delta},f^n)^{1-\delta}\ge
M(r^{(a+2\pi\delta)(1-\delta)},f^n)> M(r^a,f^n),
\]
since $0<a<1$, by \eqref{mbig}, Theorem~\ref{Har} part~(b) and Theorem~\ref{convex}.

We now define $A_n$ to be the component of $(f^n)^{-1}(A'_n)$ that contains $C$. Then
\begin{equation}\label{inner1}
f^n(\{z:|z|=r^a\})\subset \{z:|z|\le M(r^a,f^n)\}
\end{equation}
and
\begin{align}\label{outer1}
f^n(\{z:|z|=r^{b-2\pi\delta}\})&\subset \{z:|z|\ge
M(r^{b-2\pi\delta},f^n)^{(1-\delta)}\}\notag\\
&\subset \{z:|z|\ge M(r^{(b-2\pi\delta)(1-\delta)},f^n)\},
\end{align}
by \eqref{mbig}, Theorem~\ref{Har} part (b) and Theorem~\ref{convex}.
Thus
\begin{equation}\label{nested}
C\subset A_n\subset  A(r^a,r^{b-2\pi\delta})\subset A(r^a,r^b)
\end{equation}
and $A_n$ is doubly connected since any complementary component of $A_n$ must contain a zero of $f^n$ and there are no zeros of $f^n$ in $A(r^a,r^b)$. This proves the claim.

Next we take $\phi_n$ to be the conformal mapping of $A_n$ onto an annulus of the form $A(r^{a},r^{b'_n})$, where $b_n'>a$, with the property that
the inner and outer boundary components of $A_n$ map onto the inner and outer boundary components of $A(r^{a},r^{b'_n})$ respectively, and $\phi_n(r)$ is positive.


Since $r^a\ge R>1$, we deduce, by Theorem~\ref{Har} part~(b),~\eqref{C} and~\eqref{nested}, that
\[
|\phi_n(z_1)|\ge |\phi_n(z_2)|^{1-\delta}, \;\;\text{for }|z_1|=|z_2|=r^{c},\;a+4\pi\delta\le c\le b(1-5\pi\delta);
\]
note that $b(1-5\pi\delta)>a+4\pi\delta$ since $\delta<(b-a)/(4\pi+5\pi b)$. Therefore, if $R_{c}=M(r^{c},\phi_n)$, $c>0$, then
\begin{equation}\label{rho1}
\phi_n(\{z:|z|=r^{c}\})\subset \overline{A}(R_{c}^{1-\delta},R_{c}),\;\;\text{for }a+4\pi\delta\le c\le b(1-5\pi\delta).
\end{equation}

Recall that the modulus of a ring domain is preserved under a conformal mapping and is monotonic under containment, and that the modulus of $A(r_1,r_2)$ is $(1/2\pi)\log r_2/r_1$; see \cite[Chapter~4]{A}, for example. We apply these facts to $A_n(c)=A_n\cap A(r^a,r^{c})$, where $a+4\pi\delta\le c\le b(1-5\pi\delta)$, which is a ring domain by \eqref{C} and \eqref{nested}. We deduce by \eqref{nested} and \eqref{rho1} that
\[
\frac{1}{2\pi}\log\left(\frac{r^{c}}{r^{a+2\pi\delta}}\right)\le {\rm mod}\,(A_n(c))\le \frac{1}{2\pi}\log\left(\frac{r^{c}}{r^a}\right),
\]
and
\[
\frac{1}{2\pi}\log\left(\frac{R_{c}^{1-\delta}}{r^a}\right)\le {\rm mod}\,(\phi_n(A_n(c)))\le\frac{1}{2\pi}\log\left(\frac{R_{c}}{r^a}\right).
\]
Thus, since $0<\delta<1$,
\begin{equation}\label{R-rho}
r^{c-2\pi \delta}\le R_{c}\le
r^{c(1+2\delta)},\;\;\text{for }a+4\pi\delta\le c\le b(1-5\pi\delta).
\end{equation}
Hence, by \eqref{rho1},
\[
r^{(c-2\pi \delta)(1-\delta)}\le |\phi_n(z)|\le r^{c(1+2\delta)},\;\;\text{for }r^{a+4\pi\delta}\le |z|=r^{c}\le r^{b(1-5\pi\delta)},
\]
which gives \eqref{modulus}.

So $f^n$ can be factorised as the conformal mapping $\phi_n$
followed by a proper map $P_n$ from the annulus
$A(r^{a},r^{b'_n})$ to the annulus $A'_n$. Since $P_n$ maps the
outer/inner boundary components of $A(r^{a},r^{b'_n})$ onto the
outer/inner boundary components respectively of $A'_n$, we deduce
that the bounded harmonic function $\log |P_n(z)|$ takes constant
values on $\{z:|z|=r^{a}\}$ and on $\{z:|z|=r^{b'_n}\}$.
Choose $d_n>0$ and $q_n>0$ such that $d_n \log|z| + \log q_n$
takes the same values as $\log |P_n(z)|$ on these circles. Then, since the
solution to the Dirichlet problem is unique,
$\log|P_n(z)| =  d_n \log|z| + \log q_n$, for $z\in A(r^{a},r^{b'_n})$, so
$|P_n(z)|=q_n|z|^{d_n}$ in this annulus. Hence $P_n(z)= \alpha q_n z^{d_n}$ for some constant $\alpha$ such that $|\alpha|=1$, and by suitably normalizing $\phi_n$ we may assume that $\alpha=1$.

To estimate the degree $d_n$ we consider the mean
\[
\mu(\rho,f^n)=\frac{1}{2\pi}\int_0^{2\pi}\log |f^n(\rho e^{i\theta})|\,d\theta, \;\;\rho>0.
\]
By \eqref{mbig} and Theorem~\ref{Har},
\begin{equation}\label{T(r)}
(1-\delta)\log M(\rho,f^n)\le \mu(\rho,f^n)\le\log M(\rho,f^n), \;\;\text{for }\rho\in[r^{a+2\pi\delta},r^{b-2\pi\delta}].
\end{equation}
We now use the following consequence of Jensen's theorem (see
\cite[Section~3.61]{eT39}):
\begin{equation}\label{Jensen}
\mu(r,f^n)=\mu(\rho,f^n)+\int_{\rho}^r\frac{n(t)}{t}\,dt, \;\text{
for } \rho\in (0,r),
\end{equation}
where $n(t)$ is the number of zeros of $f^n$ in $\{z:|z|\le t\}$. By the representation
\[
f^n(z)=P_n(\phi_n(z))=q_n(\phi_n(z))^{d_n},\;\;\mbox{for }z\in A_n,
\]
and the argument principle, we deduce that $n(t)\le d_n$, for $0<t\le r$, and $n(t)=d_n$, for $r^a\le t\le r$. Thus, by \eqref{Jensen},
\begin{equation}\label{T(r)ineq}
\mu(r^a,f^n)+(1-a)d_n\log r= \mu(r,f^n)\le \mu(R,f^n)+d_n\log r;
\end{equation}
 recall that $r^a\ge R=R(f)$, where $R(f)>1$ is a constant suitable for Theorem~\ref{convex}. Since $\mu(r^a,f^n)\ge 0$, by~\eqref{mbig}, the equality in \eqref{T(r)ineq} and the right-hand inequality in \eqref{T(r)} give
\[
d_n\le \frac{\mu(r,f^n)}{(1-a)\log r}\le \frac{\log M(r,f^n)}{(1-a)\log r}.
\]
On the other hand, if $r$ is so large that $\log R/\log r\le 1/\sqrt{\log r}=\delta$, then, by the right-hand inequality in \eqref{T(r)} and Theorem~\ref{convex},
\[
\frac{\mu(R,f^n)}{\log r}\le \frac{\log M(R,f^n)}{\log r}\le \left(\frac{\log R}{\log r}\right)\frac{\log M(r,f^n)}{\log r}\le \delta\, \frac{\log M(r,f^n)}{\log r}.
\]
 So, by the left-hand inequality in \eqref{T(r)} and the inequality in \eqref{T(r)ineq},
\[
 (1-2\delta)\frac{\log M(r,f^n)}{\log r}\le \frac{\mu(r,f^n)-\mu(R,f^n)}{\log r}\le d_n.
 \]
This completes the proof.
\end{proof}

Next we give two covering theorems, which are based on a much weaker assumption than that in Theorems~\ref{Har} and~\ref{pol}. We show that if an entire function omits the value~$0$ in an annulus (rather than omitting the unit disc there), then the image of that annulus must cover a large annulus (much larger than the annulus given by Theorem~\ref{pol}). However, under this weaker hypothesis we cannot assert that the {\ef} behaves like a monomial within the annulus.

The proofs of these two covering theorems use the contraction property of the hyperbolic metric; see \cite[Theorem~4.1]{CG}, for example. We denote the density of the hyperbolic metric at a point $z$ in a hyperbolic domain $G$ by $\rho_G(z)$
and the hyperbolic distance between $z_1$ and $z_2$ in $G$ by
$\rho_G(z_1,z_2)$.

\begin{theorem}\label{Kiel}
There exists an absolute constant $\delta>0$ such that if $f:A(R,R')\to \C\setminus\{0\}$ is analytic, where $R'>R$, then for all
$z_1,z_2 \in A(R,R')$ such that
\begin{equation}\label{cond11}
 \rho_{A(R,R')}(z_1,z_2)<\delta\quad\text{and}\quad
|f(z_2)|\ge 2|f(z_1)|,
\end{equation}
we have
\[
f(A(R,R'))\supset \overline{A}(|f(z_1)|,|f(z_2)|).
\]
\end{theorem}
\begin{proof}
Let $A_0 = A(R,R')$ and suppose that \eqref{cond11} holds for a value of $\delta$ to be chosen. Suppose also, for a contradiction, that $w_0\in
\overline{A}(|f(z_1)|,|f(z_2)|)\setminus f(A_0)$. By Pick's theorem,
\begin{align}
\rho_{A_0}(z_1,z_2)&\ge \rho_{f(A_0)}(f(z_1),f(z_2))\notag\\
&\ge \rho_{\C\setminus \{0,w_0\}}(f(z_1),f(z_2))\notag\\
&=\rho_{\C\setminus \{0,1\}}(f(z_1)/w_0,f(z_2)/w_0).\notag
\end{align}
Let $\gamma$ be a hyperbolic geodesic  in $\C\setminus \{0,1\}$ from
$t_1=f(z_1)/w_0$ to $t_2=f(z_2)/w_0$. Then choose $\gamma'$ to be
a segment of $\gamma$ which joins $t'_1$ to $t'_2$, where
$|t'_2|=2|t'_1|$ and  $1\in A(|t'_1|,2|t'_1|)$. This is possible
by the second inequality in \eqref{cond11}.

Then $t'_1, t'_2\in \overline{A}(\frac12,2)$. The density of the
hyperbolic metric on $\C\setminus\{0,1\}$ is bounded below on
$\overline{A}(\frac12,2)$ by an absolute constant, say $2\delta>0$. Hence
\[
\rho_{A_0}(z_1,z_2)\ge \rho_{\C\setminus\{0,1\}}(t'_1,t'_2)\ge
\delta,
\]
which contradicts the first inequality in \eqref{cond11} with this value of~$\delta$.
\end{proof}

We now apply Theorem~\ref{Kiel} to the iterates of an entire function to obtain the following covering result. By considering an entire function that behaves in an annulus like a monomial, it can be seen that the result is close to best possible.

\begin{theorem}\label{parta}
Let $f$ be a {\tef}. There exist $R_0=R_0(f)> 0$ and an absolute constant
$K>1$ such that if
\begin{equation}\label{annuli1a}
 0 \notin f^n(A(R,R')), \;\;\mbox{for some } n \in \N,
\end{equation}
where
\begin{equation}\label{annuli2}
 R'/R \geq K^2\;\; \mbox{and}\;\; R \geq R_0,
\end{equation}
then
\begin{equation}\label{annuli3}
f^n(A(R,R')) \supset A(M(R,f^n),M^n(R'/K,f)).
\end{equation}
\end{theorem}
\begin{proof}
Suppose that $A_0 = A(R,R')$ is an annulus satisfying~\eqref{annuli1a} and~\eqref{annuli2} for some $K>4$. Take points $z_1, z_2$
satisfying
\[|z_1| = |z_2| = 4R'/K, \; |f^n(z_2)| = M(4R'/K,f^n) \mbox{ and } |f^n(z_1)| = m(4R'/K,f^n).
\]
We now consider different cases that might arise. First, suppose that
\[
|f^n(z_1)| \geq \frac12|f^n(z_2)|=\frac{1}{2} M(4R'/K,f^n).
\]
Since $\partial f^n(A_0) \subset f^n(\partial A_0)$ and $R'> 4R'/K$, it follows that
\begin{equation}\label{high}
f^n(A_0) \supset A(M(R,f^n), \frac{1}{2} M(4R'/K,f^n)).
\end{equation}

Now suppose that
\[
M(4R'/K,f^n)=|f^n(z_2)|\ge 2|f^n(z_1)|.
\]
By assumption, we have $4RK\le 4R'/K<R'$, and so if $K$ is a
sufficiently large absolute constant, then $d_{A_0}(z_1,z_2) <
\delta$, where $\delta$ is the constant given in
Theorem~\ref{Kiel}. Thus, by Theorem~\ref{Kiel},
\begin{equation}\label{low}
f^n(A_0) \supset \overline{A}(|f^n(z_1)|, |f^n(z_2)|) =\overline{A}(|f^n(z_1)|,
M(4R'/K,f^n)).
\end{equation}
If $|f^n(z_1)| \leq M(R,f^n)$, then it follows directly from~\eqref{low} that
\begin{equation}\label{lowa}
f^n(A_0) \supset A(M(R,f^n), M(4R'/K,f^n)).
\end{equation}
On the other hand, if $M(R,f^n) < |f^n(z_1)|$, then~\eqref{lowa}
also follows from~\eqref{low} using the fact that $\partial
f^n(A_0) \subset f^n(\partial A_0)$.

The conclusion now follows from \eqref{high} and \eqref{lowa} because
\[
\frac12 M(4R'/K,f^n)\ge \frac12 M^n(2R'/K,f)\ge M^n(R'/K,f),
\]
for $R'$ sufficiently large, by Lemma~\ref{Ere} and~\eqref{2r}.
\end{proof}

\section{Proofs of Theorems~\ref{main1a} and \ref{main1}}
\setcounter{equation}{0}
Let $f$ be a {\tef} with a {\mconn} {\wand} $U$, let $z_0 \in U$ and, for $n \in \N$, let $r_n = |f^n(z_0)|$.
In order to prove Theorems~\ref{main1a} and \ref{main1} we consider the functions defined by
\begin{equation}\label{hn}
 h_n(z) = \frac{\log |f^n(z)|}{\log r_n}, \; z \in \overline U, \; n \in \N.
\end{equation}
Without loss of generality, we can assume that, for $z \in \overline U$, $n \in \N$, we have $|f^n(z)|>1$ and hence $h_n$ is a positive harmonic function in $U$ with a continuous extension to $\overline U$.

Since $h_n(z_0) = 1$, for $n \in \N$, it follows from Harnack's
theorem \cite[Theorem~1.3.10]{tR95} that there exists a sequence
$(n_k)$ such that $h_{n_k}$ converges locally uniformly in $U$
and the function $h$ defined by
 \begin{equation}\label{hlimit}
 h(z) = \lim_{k \to \infty} h_{n_k}(z)
 \end{equation}
  is a positive harmonic function on $U$. In order to prove Theorem~\ref{main1a}, we must show that $h$ is non-constant and that the whole sequence $h_n$ converges to $h$ in~$U$.

\begin{lemma}\label{nonc}
Let $h$ be the harmonic function on $U$ defined by~\eqref{hlimit}. Then $h$ is non-constant.
\end{lemma}
\begin{proof}
It follows from Theorem A and Lemma~\ref{Ere} that, if $m$ is sufficiently large, then there exist $z_1, z_2 \in U$ such that
\[
 |f^n(f^m(z_2))| \geq M^n(|2f^m(z_1)|,f)\geq M(|2f^m(z_1)|,f^n),\;\;\mbox{for } n \in \N.
\]
Clearly
\[
 |f^n(f^m(z_1))| \leq M(|f^m(z_1)|,f^n),\;\;\mbox{for } n \in \N,
\]
and so, by~\eqref{2r},
\[
\liminf_{n \to \infty} \frac{h_{n+m}(z_2)}{h_{n+m}(z_1)} = \liminf_{n \to \infty} \frac{\log |f^{n+m}(z_2)|}{\log |f^{n+m}(z_1)|} \geq \liminf_{n \to \infty} \frac{\log M(2|f^m(z_1)|,f^n)}{\log M(|f^m(z_1)|,f^n)} > 1.
\]
Thus $h(z_2) > h(z_1)$, so $h$ is non-constant.
\end{proof}

We now show that the whole sequence $h_n$ converges to $h$ in $U$. In order to do this, we first prove two lemmas.
The first concerns the sequence
\begin{equation}\label{gn}
 g_n(z) = \frac{\log f^n(z)}{\log r_n}, \; \;z \in V, \; n \in \N,
\end{equation}
where $V \subset U$ is a simply connected domain with $z_0 \in
V$, and the branch of the logarithm is chosen so that $|\arg
f^n(z_0)| \leq \pi$. Then, for $n \in \N$, $g_n$ is analytic
in~$V$ and $\Re g_n = h_n$.

\begin{lemma}\label{seq}
Let $V \subset U$ be a simply connected domain with $z_0 \in V$
and let $h_n$, $h$ and $g_n$ be the functions defined
by~\eqref{hn},~\eqref{hlimit} and~\eqref{gn} respectively. If
$(h_{m_k})$ is any subsequence of $(h_n)$ such that
\[
h_{m_k} \to h \mbox{ locally uniformly in } U,
\]
then
\begin{itemize}
\item[(a)] there exists an analytic function $g$ on $V$ with $\Re g = h$ such that $g_{m_k} \to g$ locally uniformly in $V$ and
\item[(b)] for any continuum $C \subset V$, there exists $\eps > 0$ such that, for large $k$,
\[
U_{m_k} \supset A(\min_{z \in C} |f^{m_k}(z)|r_{m_k}^{-\eps/8}, \max_{z \in C}|f^{m_k}(z)|r_{m_k}^{\eps/8}).
\]
\end{itemize}
\end{lemma}
\begin{proof}
First, $(g_n)$ is a normal family in $V$ by Montel's theorem because $(h_n)$
is locally uniformly bounded there by Harnack's inequality
\cite[Theorem~1.3.1]{tR95}, since $h_n(z_0)=1$ for $n\in\N$.

Let $v$ be the harmonic conjugate of $h$ on $V$, chosen so that $v(z_0)=0$, and put
\[
g(z)=h(z)+iv(z).
\]
Then $g$ is analytic in $V$ and
\begin{equation}\label{g-diff}
g_n(z)-g(z)=h_n(z)-h(z)+i(\arg f^n(z)/\log r_n-v(z)).
\end{equation}
Now choose $\rho>0$ such that $\{z:|z-z_0|\le 2\rho\}\subset V$. Then by the Borel-Carath\'eodory inequality (see \cite[page~20]{gV}) we have, for $n\in\N$,
\begin{equation}\label{Bor-Car}
\max_{|z-z_0|=\rho}|g_n(z)-g(z)|\le 2\max_{|z-z_0|\le 2\rho}|h_n(z)-h(z)|+3|g_n(z_0)-g(z_0)|.
\end{equation}
Since $h_{m_k}\to h$ locally uniformly in $U$, $|\arg f^n(z_0)|\le \pi$, for $n\in\N$, and $v(z_0)=0$, it follows from \eqref{g-diff} and \eqref{Bor-Car} that $g_{m_k}\to g$ uniformly in $\{z:|z-z_0|\le \rho\}$. Since $(g_n)$ is normal in $V$ it follows that $g_{m_k}\to g$ locally uniformly in $V$ by Vitali's theorem. This proves part~(a).

Now let $C \subset V$ be a continuum and for $\delta>0$ let
$C_{\delta}$ denote the $\delta$-neighbourhood of $C$. If
$\delta>0$ is such that $\overline{C_{\delta}}\subset V$, then
there exists $\eps>0$ such that
\[
g(C_{\delta})\supset g(C)_{\eps}.
\]
Thus for large $k$ we have, by part~(a),
\[
g_{m_k}(C_{\delta})\supset g(C)_{\eps/2}\supset g_{m_k}(C)_{\eps/4},
\]
and hence
\begin{align}
U_{m_k}&\supset f^{m_k}(C_{\delta})\notag\\
&= \exp\left((\log r_{m_k}) g_{m_k}(C_{\delta})\right)\notag\\
&\supset \exp\left((\log r_{m_k}) g_{m_k}(C)_{\eps/4}\right)\notag\\
&\supset A(\min_{z \in C} |f^{m_k}(z)|r_{m_k}^{-\eps/8}, \max_{z \in C}|f^{m_k}(z)|r_{m_k}^{\eps/8}),\notag
\end{align}
provided that $(\eps/8)\log r_{m_k}>\pi$. This completes the proof of part~(b).
\end{proof}

Our second lemma concerns the relationship between the values taken by $h_n$ and $h_{n+m}$ for $n,m \in \N$.

\begin{lemma}\label{delta}
Let $z \in U$ and, for $n \in \N$, let $h_n$ be the function defined by~\eqref{hn} and $\delta_n = 1/\sqrt{\log r_n}$. There exists $N \in \N$ such that
\begin{itemize}
\item[(a)] if, for some $n \geq N$,
\[
U_n \supset A(r_n^{1-2\pi\delta_n},r_n^{1 + 2\pi\delta_n}),
\]
then
\[
r_{n+m} \geq M(r_n,f^m)^{1 - \delta_n},\;\;\mbox{for } m \in \N;
\]
\item[(b)] if, for some $n \geq N$,
\begin{equation}\label{hbig}
h_n(z) \geq 1 \;\mbox{ and }\; U_n \supset A(r_n^{h_n(z)-2\pi\delta_n},r_n^{h_n(z) + 2\pi\delta_n}),
\end{equation}
then
\[
h_{n+m}(z) \geq h_n(z)(1 - \delta_n),\;\;\mbox{for } m \in \N;
\]
\item[(c)] if, for some $n \geq N$,
\begin{equation}\label{hsmall}
h_n(z) \leq 1 \;\mbox{ and }\; U_n \supset A(r_n^{h_n(z)-2\pi\delta_n},r_n^{h_n(z) + 2\pi\delta_n}),
\end{equation}
then
\[
h_{n+m}(z) \leq h_n(z)(1 + 2\delta_n),\;\;\mbox{for } m \in \N;
\]
\item[(d)] if $z' \in U$ and, for some $n \geq N$,
\begin{equation}\label{hcomp}
h_n(z)/ h_n(z') \geq 1 \;\mbox{ and }\; U_n \supset A(r_n^{h_n(z)-2\pi\delta_n},r_n^{h_n(z)+2\pi\delta_n}),
\end{equation}
then
\[
h_{n+m}(z)/h_{n+m}(z') \geq  (1 - \delta_n) h_n(z)/h_n(z'),\;\;\mbox{for } m \in \N.
\]
\end{itemize}
\end{lemma}
\begin{proof}
Suppose that~\eqref{hbig} is satisfied for some $z \in U$, $n \in \N$. Since $|f^m| > 1$ in $U_n$, for all $m \in \N$, it follows from Theorem~\ref{Har} part (b) applied to $g = f^m$ that, for $m\in\N$ and sufficiently large $n$,
\begin{equation}\label{fHar}
\log |f^m(f^n(z))| \geq (1 - \delta_n)\log M(|f^n(z)|,f^m).
\end{equation}
Since $h_n(z_0)=1$, part~(a) now follows by taking $z = z_0$.

To prove part~(b), we recall from \eqref{hn} that $\log |f^n(z)| = h_n(z)\log r_n$, for $n \in \N$, $z \in U$, and so it follows from~\eqref{fHar} and Theorem~\ref{convex} that, if~\eqref{hbig} is satisfied for some $z \in U$, $n \in \N$, where $n$ is sufficiently large, then, for $m \in \N$,
\begin{eqnarray*}
h_{n+m}(z) \log r_{n+m} & = & \log|f^m(f^n(z))| \geq (1 - \delta_n)\log M(r_n^{h_n(z)},f^m)\\
 & \geq & (1 - \delta_n)h_n(z) \log M(r_n,f^m) \geq (1 - \delta_n)h_n(z) \log r_{n+m}.
\end{eqnarray*}
Part~(b) now follows.

To prove part~(c), we suppose that~\eqref{hsmall} is satisfied for some $z \in U$, $n \in \N$. Then it follows from Theorem~\ref{convex} and part~(a) that, for $m\in\N$ and sufficiently large $n$,
\begin{eqnarray*}
h_{n+m}(z) \log r_{n+m} &=& \log |f^m(f^n(z))| \leq \log M(|f^n(z)|,f^m)\\
&=&  \log M(r_n^{h_n(z)},f^m)\leq h_n(z) \log M(r_n,f^m)\\
 &\leq& h_n(z) \log r_{n+m}/ (1 - \delta_n)\leq h_n(z) \log r_{n+m} (1 + 2\delta_n).
\end{eqnarray*}
This proves part~(c).

Finally, we suppose that~\eqref{hcomp} is satisfied for some $z, z' \in U$, $n \in \N$. Then it follows from Theorem~\ref{Har} part (b) and Theorem~\ref{convex}, with $r=|f^n(z')|$ and $r^c=|f^n(z)|$, that, for $m\in\N$ and sufficiently large $n$,
\begin{eqnarray*}
\frac{h_{n+m}(z)}{h_{n+m}(z')} & = & \frac{\log |f^{n+m}(z)|}{\log |f^{n+m}(z')|}\\
& \geq & (1 - \delta_n)\frac{\log M(|f^n(z)|,f^m)}{\log M(|f^n(z')|,f^m)}\\
&  \geq& (1 - \delta_n)\frac{\log |f^n(z)|}{\log |f^n(z')|} = (1 - \delta_n)\frac{h_n(z)}{h_n(z')}.
\end{eqnarray*}
This proves part~(d).
\end{proof}
\begin{proof}[Proofs of Theorems~\ref{main1a} and \ref{main1}] Recall that, for the functions $h_n$ defined in~\eqref{hn}, we know that there exists a subsequence $(n_k)$ such that the functions $h_{n_k}$ converge locally uniformly in $U$ to a positive non-constant harmonic function $h$. Thus the hypotheses of Lemma~\ref{seq} hold for the sequence $(n_k)$. It then follows from Lemma~\ref{seq} part~(b) that we can apply Lemma~\ref{delta} parts (b) and (c) to $U_{n_k}$, provided that $k$ is sufficiently large, and so deduce that the whole sequence $h_n$ converges to $h$ locally uniformly in $U$. This completes the proof of Theorem~\ref{main1a}.

The result of Theorem~\ref{main1} follows from Theorem~\ref{main1a} and Lemma~\ref{seq} part~(b) applied to the whole sequence $h_n$ with $C=\{z_0\}$.
\end{proof}

\section{Dynamics in a multiply connected wandering domain}
\setcounter{equation}{0}

Let $f$ be a {\tef} with a {\mconn} {\wand} $U$, let $z_0 \in U$ and recall that there exists $\alpha > 0$ such that, for large $n \in \N$,
the maximal annulus in $U_n$, centred at $0$, that contains $f^n(z_0)$ is of the form
 \[
 B_n = A(r_n^{a_n}, r_n^{b_n})
 , \mbox{ where } r_n = |f^n(z_0)|,  \; 0 < a_n < 1 - \alpha < 1 + \alpha < b_n.
 \]
 Also, recall that
 \[
 C_n = A\left(r_n^{a_n+2\pi\delta_n},r_n^{b_n(1-3\pi\delta_n)}\right), \mbox{ where } \delta_n = 1/\sqrt{\log r_n}.
 \]
In this section we prove several results concerning the dynamics of $f$ in the annuli $B_n$ and $C_n$. The proofs use the harmonic functions $h_n$ defined in \eqref{hn} and
studied in the previous section. Recall from Theorem~\ref{main1a} that
\begin{equation}\label{lim}
 h(z) = \lim_{n\to \infty} h_n(z),\;\;\mbox{for } z \in U,
\end{equation}
defines a non-constant positive harmonic function in $U$.

We begin by proving Theorem~\ref{main2a} which shows that the union of the annuli $C_n$ acts as an absorbing set for the dynamics of $f$ in $U$.

\begin{proof}[Proof of Theorem~\ref{main2a}]
Theorem~\ref{main2a} states that, if $C$ is a compact subset of $U$, then $f^n(C) \subset
C_n$, for large $n \in \N$. In order to prove this we cover $C$ by a
finite number of closed discs $D_1,\ldots, D_p$ in $U$, and
then, for $j = 1,\ldots,p$, join $D_j$ to $\{z_0\}$ by a simple curve $L_j$ in $U$.
It follows from~\eqref{lim} that we can apply Lemma~\ref{seq} part~(b) to each continuum $D_j \cup L_j$,
$j=1,\ldots,p$, with a suitable simply connected domain $V_j
\subset U$ containing $D_j \cup L_j$, to show that there exists $\eps>0$ such that, for large $n \in \N$,
\[
U_n=f^n(U) \supset A(\min_{z \in C}|f^n(z)|r_n^{-\eps/8}, \max_{z \in C}|f^n(z)|r_n^{\eps/8})
\]
and hence, since $z_0\in C$,
\begin{equation}\label{Can}
f^n(C) \subset A(r_n^{a_n+\eps/8},r_n^{b_n-\eps/8}).
\end{equation}
Now let $H = \max_{z \in C} h(z)$. If $b_n \le 2H$, then for sufficiently large $n$ we have
\[
a_n + \eps/8 > a_n + 2\pi \delta_n\quad \mbox{and}\quad  b_n - \eps/8 < b_n(1 - 3\pi \delta_n),
\]
so it follows from~\eqref{Can} that $f^n(C) \subset C_n$.

If $b_n > 2H$, then for sufficiently large $n$ we have $a_n + \eps/8 > a_n + 2\pi \delta_n$ and
\[
b_n(1-3 \pi \delta_n) > b_n/2 >H = \max_{z \in C}h_n(z) = \max_{z \in C}\frac{\log |f^n(z)|}{\log r_n},
\]
so $f^n(C) \subset C_n$, by~\eqref{Can} again. This completes the proof.
\end{proof}

The next result implies that in $C_n$ the minimum modulus is very close to the maximum modulus. This is a key result that we use often and it follows immediately from Theorem~\ref{Har} by taking $g = f^m$, $a = a_n$, $b = b_n$ and $\delta = \delta_n$.

\begin{theorem}\label{main2b}
Let $f$ be a {\tef} with a {\mconn} {\wand} $U$, let $z_0 \in U$ and, for large $n \in \N$, let $r_n$, $a_n$ and $b_n$ be defined as in \eqref{Bndef}.
\begin{itemize}
\item[(a)]
If $m \in \N$ and $\eps \in \left(\pi/\log r_n, (b_n - a_n)/2\right)$, where $n \in \N$ is sufficiently large, then
\[
\log m(\rho,f^m) \geq \left(1 - \frac{2\pi}{\eps \log r_n}\right) \log M(\rho,f^m),\;\;\mbox{for } \rho \in [r_n^{a_n + \eps}, r_n^{b_n - \eps}].
\]
\item[(b)] In particular, if $\delta_n = 1/\sqrt{\log r_n}$, then, for $m \in \N$ and large $n \in \N$,
\[
\log m(\rho,f^m) \geq (1 - \delta_n) \log M(\rho,f^m),\;\;\mbox{for } \rho \in [r_n^{a_n + 2\pi \delta_n}, r_n^{b_n - 2\pi \delta_n}].
\]
\end{itemize}
\end{theorem}
Another application of Theorem~\ref{main2b} part~(a) is given in \cite[Corollary~4.1]{wB11}.

The next result states that, for large $n \in \N$, the iterates of $f$ behave like monomials inside $C_n$. This follows immediately from Theorem~\ref{pol}.

\begin{theorem}\label{poly}
Let $f$ be a {\tef} with a {\mconn} {\wand} $U$, let $z_0 \in U$ and, for large $n \in \N$, let $r_n$, $a_n$, $b_n$, $\delta_n$, $B_n$ and $C_n$  be defined as in \eqref{Bndef} and~\eqref{Cn}.

There exists $N\in\N$ such that for $n \ge N$ and $m \in \N$, the preimage under~$f^m$ of $A(M(r^{a_n},f^m),M(r^{(b_n-2\pi\delta_n)(1-\delta_n)},f^m))$ has a doubly connected component $A_{n,m}$ such that
\[
C_n \subset A_{n,m} \subset B_n
\]
and $f^m$ has no critical points in $A_{n,m}$. Moreover, for $n \ge N$ and $m \in \N$, we can write $f^m(z) = P_{n,m}(\phi_{n,m}(z))$, $z \in A_{n,m}$, where
\begin{itemize}
\item[(a)] $\phi_{n,m}$ is a conformal map on $A_{n,m}$ which satisfies
\[
|z|^{1-\delta_n-2\pi\delta_n\log r_n/\log |z|}\le|\phi_{n,m}(z)|\le |z|^{1 + 2\delta_n}, \;\;\mbox{for } r_n^{a_n+4\pi\delta_n}\le |z|\le r_n^{b_n(1-5\pi\delta_n)};
\]
\item[(b)] $P_{n,m}(z) = q_{n,m}z^{d_{n,m}}$, where $q_{n,m}>0$, and the degree $d_{n,m}$ is the number of zeros of~$f^m$ in $\{z:|z|<r_n\}$, which satisfies
\[
(1 - 2 \delta_n) \frac{\log M(r_n,f^m)}{\log r_n} \leq d_{n,m} \leq \frac{\log M(r_n,f^m)}{(1 - a_n) \log r_n}.
\]
\end{itemize}
\end{theorem}
In particular, Theorem~\ref{poly} implies that for $n \ge N$ and $m \in \N$, the function $f^m$ has no critical points in $C_n$; thus we have proved Theorem~\ref{nocrits}.

We conclude this section by showing how the sequences of annuli $(B_n)$ and $(C_n)$ relate to each other.

\begin{theorem}\label{cover}
Let $f$ be a {\tef} with a {\mconn} {\wand} $U$, let $z_0 \in U$ and, for large $n \in \N$, let $r_n$, $a_n$, $b_n$, $\delta_n$, $B_n$ and $C_n$  be defined as in \eqref{Bndef} and~\eqref{Cn}.

There exists an absolute constant $K>1$ such that, for $m \in \N$ and large $n \in \N$,
 \begin{itemize}
 \item[(a)]
  $M(r_n,f^m) \geq r_{n+m} \geq M(r_n,f^m)^{1- \delta_n}$;
\item[(b)]
$f^m(B_n) \cap B_{n+m} \supset A(M(r_n^{a_n},f^m), M(r_n^{b_n}/K,f^m)) \supset A(r_{n+m}^{a_n + \delta_n}, r_{n+m}^{b_n - \delta_n})$;
\item[(c)] $f^m(C_n) \subset B_{n+m}.$
\end{itemize}
\end{theorem}
\begin{proof}
{\it (a)} Since $r_n = |f^n(z_0)|$, for $n \in \N$, it is clear that, for $m,n \in \N$, we have $r_{n+m} \leq M(r_n,f^m)$. Also, it follows from Theorem~\ref{main2b} part (b) that, for $m \in \N$ and large $n \in \N$, $r_{n+m} \geq M(r_n,f^m)^{1- \delta_n}$.

{\it (b)} We now recall from Theorem~\ref{main1} that there exists $\alpha > 0$ such that
\begin{equation}\label{alpha}
a_n \leq 1 - \alpha \;\mbox{ and }\; b_n \geq 1 + \alpha, \mbox{
for large } n \in \N.
\end{equation}
In particular, $r_n^{b_n}/r_n^{a_n}\ge r_n^{2\alpha}$. Thus, by
Theorem~\ref{parta} there exists an absolute
constant $K>0$ such that, for $m \in \N$ and large $n \in N$,
\begin{align}\label{fmBn}
 f^m(B_n) &\supset A(M(r_n^{a_n},f^m), M^m(r_n^{b_n}/K,f))\\
 &\supset A(M(r_n^{a_n},f^m), M(r_n^{b_n}/K,f^m))\notag.
\end{align}
    It follows from Theorem~\ref{convex},~part (a) of this result and~\eqref{alpha} that, for $m \in \N$ and large $n \in N$,
\begin{equation}\label{a1}
M(r_n^{a_n},f^m) \leq M(r_n,f^m)^{a_n} \leq
r_{n+m}^{a_n/(1-\delta_n)} \leq r_{n+m}^{a_n(1+\delta_n/(1-
\alpha))} \leq r_{n+m}^{a_n +\delta_n}.
\end{equation}
Also, it follows from~\eqref{alpha} that if $n$ is sufficiently large, then $b_n - \delta_n > 1$. So, by Theorem~\ref{convex} and part (a) of this result, for $m \in \N$ and large $n \in N$,
\begin{equation}\label{b1}
 M(r_n^{b_n}/K,f^m) \geq M(r_n^{b_n - \delta_n}, f^m) \geq M(r_n,f^m)^{b_n - \delta_n} \geq r_{n+m}^{b_n - \delta_n}.
\end{equation}
It follows from~\eqref{fmBn},~\eqref{a1} and~\eqref{b1} that, for $m \in \N$ and large $n \in N$,
\[
f^m(B_n) \supset A(M(r_n^{a_n},f^m), M^m(r_n^{b_n}/K,f)) \supset A(r_{n+m}^{a_n +\delta_n},r_{n+m}^{b_n - \delta_n})
\]
with $a_n+\delta_n<b_n-\delta_n$, and hence that
\begin{equation}\label{Bnm1}
B_{n+m} \supset A(M(r_n^{a_n},f^m), M^m(r_n^{b_n}/K,f)).
\end{equation}

{\it (c)} It follows from Theorem~\ref{poly} and the definition
of the annulus $A_{n,m}$ in the statement of Theorem~\ref{poly} that, for
$m \in \N$ and large $n \in N$,
\[
f^m(C_n) \subset f^m(A_{n,m}) = A(M(r_n^{a_n},f^m),M(r_n^{(b_n - 2\pi\delta_n)(1-\delta_n)},f^m)).
\]
Thus, by~\eqref{Bnm1}, for $m \in \N$ and large $n \in N$, we have $f^m(C_n)\subset B_{n+m}$.
This completes the proof.
\end{proof}

\section{Geometric properties of multiply connected wandering domains}
\setcounter{equation}{0}
In this section we consider the geometric properties of the images of a {\mconn} {\wand} of an {\ef} and prove Theorem~\ref{main2c}. We begin by recalling some basic properties of the boundary components of such a domain which will be used several times. First recall that for any bounded domain $U$ the outer boundary component of $U$, denoted by $\partial_{\infty}U$, is the boundary of the unbounded component of $\C\setminus U$ and the inner boundary component of $U$, denoted by $\partial_{0}U$, is the boundary of the component of $\C\setminus U$ that contains~$0$, if there is one.

Now let $U$ be a {\mcwd}. Since $f^n:U\to U_n$ is a proper map, each boundary component of $U$ is mapped by $f^n$, $n\in\N$, to a boundary component of $U_n=f^n(U)$. In particular, since $\partial U_n\subset f^n(\partial U)$, the outer boundary component of $U$ maps to the outer boundary component of $U_n$. Since there exists $N\in\N$ such that the image domains $U_n=f^n(U)$, $n\ge N$, surround~0, $U_n$ has an inner boundary component for $n\ge N$, and~$f$ maps the inner boundary component of such a $U_n$ to the inner boundary component of~$U_{n+1}$.


We now prove Theorem~\ref{main2c}.

\begin{proof}[Proof of Theorem~\ref{main2c}]
Let $B_n = A(r_n^{a_n},r_n^{b_n})$ be defined as in~\eqref{Bndef}. We begin by showing that the sequences $(a_n)$ and $(b_n)$ are convergent.

It follows from Theorem~\ref{cover} part~(b) that, for $m\in \N$ and large $n \in \N$, we have
\[
a_{n+m} \leq a_n +\delta_n\;\text{ and }\; b_{n+m} \geq b_n - \delta_n.
\]
Together with~\eqref{alpha}, this implies that
\begin{equation}\label{ab}
a_n \to a \in [0,1) \;\;\mbox{and}\;\; b_n \to b \in (1,\infty],
\mbox{ as } n \to \infty.
\end{equation}

Next we prove part~(b). Clearly $\underline{b}_{\,n} \geq b_n$ and $\underline{b}_{\,n}/\underline{a}_{\,n} \geq b_n/a_n$ for $n\in\N$.
We now show that
\[b \geq \underline{b}_{\,n}\;\;\text{and}\;\;b/a\geq \underline{b}_{\,n}/\underline{a}_{\,n},\]
for sufficiently large $n \in\N$. To do this we take $N\in\N$ so large that $r_n^{\underline{a}_{\,n}}\ge R$ for $n\ge N$, where $R>0$ is the constant in Theorem~\ref{convex}. Then take $\eps > 0$ so small that $a_n +\eps < 1<b_n-\eps$ for $n\ge N$, which is possible by Theorem~\ref{main1}.

Fix $n\ge N$ and for $m \in \N$, let $w_m$ denote a point satisfying
\begin{equation}\label{zm}
|w_m| = r_n^{\underline{b}_{\,n} - \eps} \;\mbox{ and }\; |f^m(w_m)| =
M(r_n^{\underline{b}_{\,n} - \eps},f^m).
\end{equation}
 Then there exist $w_0$ and a subsequence $(m_k)$ with
 \begin{equation}\label{z0}
 |w_0| = r_n^{\underline{b}_{\,n}- \eps}\; \mbox{ and }\; w_0 = \lim_{k \to \infty} w_{m_k}.
 \end{equation}
We claim that $w_0 \in U_n$. If not, then $w_0$ belongs to a bounded component of the complement of $U_n$, so there exists $M \in \N$ such that $|f^M(w_0)| < r_{n+M}$ (because such a component lies in the interior of a Jordan curve in $U$ and so, by Theorem~\ref{main2a}, it must map under $f^m$ for large $m$ to a component of $\C\setminus U_{n+m}$ that is surrounded by $B_{n+m}$). Thus, for sufficiently large $k$, $|f^M(w_{m_k})| < r_{n+M}$ and hence
\[
 |f^{m_k}(w_{m_k})| < M(r_{n+M},f^{m_k-M}) \leq M(r_n,f^{m_k}) < M(r_n^{\underline{b}_{\,n}- \eps},f^{m_k}).
\]
This, however, contradicts~\eqref{zm} and so $w_0$ must be in $U_n$ as claimed.

It now follows from~\eqref{zm}, \eqref{z0} and Theorem~\ref{main2a} that, for $k$ sufficiently large,
\begin{equation}\label{Bnmk}
B_{n+m_k} \supset \{z:|z| = M(r_n^{\underline{b}_{\,n}- \eps},f^{m_k})\}.
\end{equation}
It follows from Theorem~\ref{convex} and Theorem~\ref{cover} part (a) that, for $k \in \N$,
\[
M(r_n^{\underline{b}_{\,n}- \eps},f^{m_k}) \geq M(r_n,f^{m_k})^{\underline{b}_{\,n} - \eps} \geq r_{n+m_k}^{\underline{b}_{\,n}- \eps}.
\]
Together with~\eqref{Bnmk}, this implies that, if $k$ is sufficiently large, then $b_{n+m_k} \geq \underline{b}_{\,n} - \eps$ and so $b \geq \underline{b}_{\,n} - \eps$. Since $\eps$ can be chosen to be arbitrarily small, this implies that $b \geq \underline{b}_{\,n}$ as required.

Next, with the same value of $n$ and $\eps > 0$ so small that $a_n + \eps < 1<b_n-\eps$, we let $z' \in U_n$ satisfy $|z'| = r_n^{\underline{a}_{\,n} + \eps}$.
Then, by Theorem~\ref{main2a},
\begin{equation}\label{Bnm}
B_{n+m} \supset \{z:|z| = |f^m(z')|\}, \mbox{ for large } m \in \N.
\end{equation}

It also follows from Theorem~\ref{convex} that, for $k \in \N$,
\[
M(r_n^{\underline{b}_{\,n} - \eps},f^{m_k}) \geq M(r_n^{\underline{a}_{\,n}+ \eps},f^{m_k})^{(\underline{b}_{\,n} - \eps)/(\underline{a}_{\,n}+\eps)} \geq |f^{m_k}(z')|^{(\underline{b}_{\,n} - \eps)/(\underline{a}_{\,n}+\eps)}.
\]
Together with~\eqref{Bnmk},~\eqref{Bnm} and Theorem~\ref{main2a}, this implies that, if $k$ is sufficiently large, then $b_{n+m_k}/a_{n+m_k} \geq (\underline{b}_{\,n}-
\eps)/(\underline{a}_{\,n} + \eps)$ and so, since~$\eps$ can be chosen
to be arbitrarily small, $b/a \geq
\underline{b}_{\,n}/\underline{a}_{\,n}$. This completes the proof of part~(b)(i).

To prove part~(b)(ii), we note that it follows from part~(b)(i)
that, for large~$n$, we have $b \geq \underline{b}_{\,n} \geq b_n$.
Since $b_n \to b$ as $n \to \infty$, this implies that
$\underline{b}_{\,n} \to b$ as $n \to \infty$. Similarly, for large
$n$, we have $b/a \geq \underline{b}_{\,n}/\underline{a}_{\,n} \geq
b_n/a_n$. Since $b_n/a_n \to b/a$ as $n \to \infty$, this implies
that $\underline{b}_{\,n}/\underline{a}_{\,n} \to b/a$ as $n \to \infty$. Thus $\underline{a}_{\,n} \to a$ as $n \to \infty$ because $\underline{b}_{\,n} \to b$ as $n \to \infty$. The fact that $\underline{a}_{\,n}\le \overline{a}_n\le a_n$, for large $n$, then implies that $\overline{a}_{n} \to a$ as $n \to \infty$.
\end{proof}

\section{Properties of the harmonic function~$h$}
\setcounter{equation}{0}

In this section we consider the harmonic function $h$ defined
in~\eqref{hdef}. First we prove Theorem~\ref{h-measure} and then
we study the properties of the level curves of $h$. Recall that
$h_n\to h$ locally uniformly in $U$, where the functions $h_n$
are defined in $\overline U$ by \eqref{hn}.

\begin{proof}[Proof of Theorem~\ref{h-measure}]
To prove part~(a) we must show that $a = \inf_{z \in U} h(z)$ and $b = \sup_{z \in U} h(z)$. We begin by noting from Theorem~\ref{main2a} that, if $z \in U$, then for large $n \in \N$,
\[
r_n^{a_n} \leq |f^n(z)| \leq r_n^{b_n},\;\; \mbox{so} \;\; a_n
\leq h_n(z) \leq b_n.
\]
Thus
\begin{equation}\label{hbounds}
 a \leq h(z) \leq b,\;\;\mbox{for } z \in U.
\end{equation}
Now suppose that $1 < b' < b'' < b$.  It follows from~\eqref{ab} that, for large $n \in \N$, we have
\[
b_n > b'' + 2\pi \delta_n \;\;\mbox{and}\;\; b''(1 - \delta_n) >
b'.
\]
So, by Lemma~\ref{delta} part~(b), there exist $z \in U$ and $n \in \N$ with $h_n(z) = b''$ and
\[
h_{n+m}(z) \geq h_n(z) (1 - \delta_n) = b''(1 - \delta_n) > b',\;\;\mbox{for } m \in \N.
\]
Thus
\[
 \sup_{z \in U} h(z) \geq b',\;\;\mbox{for } b' \in (1,b).
\]
Together with~\eqref{hbounds}, this implies that $b = \sup_{z \in
U} h(z)$. The proof that $a = \inf_{z \in U} h(z)$ is similar, but
uses Lemma~\ref{delta} part~(c). This completes the proof of
part~(a).

We next show that, whatever the value of $b$, the function~$h$ has a continuous extension to $\partial U\setminus\partial_{\infty} U$ with
\begin{equation}\label{ha}
h(\zeta)=a,\;\;\text{for  }\zeta\in \partial U\setminus\partial_{\infty} U.
\end{equation}
Let $K$ be any component of $\partial U\setminus\partial_{\infty} U$. Then there exists a Jordan curve $\gamma$ in $U$ that contains $K$ in its interior int\,$(\gamma)$.

For large $n$ we have $f^n(\gamma)\subset C_n\subset B_n$, by Theorem~\ref{main2a}, so
\[
f^n({\rm int}\,(\gamma))\subset \{z:|z|<r_n^{b_n}\},
\]
and thus, by the definitions of $a_n$ and $\underline{a}_{\,n}$,
\[
f^n({\rm int}\,(\gamma)\cap\partial U)\subset \overline{A}(r_n^{\underline{a}_{\,n}},r_n^{a_n}).
\]
Since $a_n\to a$ as $n\to\infty$ and $\underline{a}_{\,n}\to a$ as $n\to\infty$, by Theorem~\ref{main2c}, we deduce that $h_n\to a$ uniformly on ${\rm int}\,(\gamma)\cap\partial U$. Clearly $h_n\to h$ as $n\to\infty$ uniformly on $\gamma$. Thus $h_n$ converges uniformly on $\partial({\rm int}\,(\gamma)\cap U)$ and hence also on $\overline{{\rm int}\,(\gamma)\cap U}$. In particular, $h=\lim_{n\to\infty}h_n$ has a continuous extension to ${\rm int}\,(\gamma)\cap\partial U$, and hence to $K$, with value $a$ there. This proves part~(b).

Part~(c) follows from part~(b). Each point of $\partial_{\infty}U$ is
regular for the Dirichlet problem because $\partial_{\infty}U$ is a continuum, and each
point of $\partial U\setminus \partial_{\infty}U$ is regular
because, by part~(b), the function $u(z)=a-h(z)$ is harmonic and negative in~$U$ with
limit~$0$ at each point of $\partial U\setminus
\partial_{\infty}U$, and so is a `barrier' at each point
of this set; see \cite[page~88]{tR95} for details of
regular boundary points and barriers.

To prove part~(d) recall that, for $z\in U$, the function $\omega(z,\partial_{\infty}U,U)$ is defined by the Perron method (see \cite[Chapter~4]{tR95}) as
\[
\omega(z,\partial_{\infty}U,U)=\sup\{u(z): u\in {\mathcal P}\},
\]
where ${\mathcal P}$ is the Perron family associated with the characteristic function of $\partial_{\infty} U$:
\[
{\mathcal P}=\{u: u\text{ is subharmonic in }U,\; \limsup_{z\to \zeta} u(z)\le \chi^{_{_{}}}_{\partial_{\infty}U}(\zeta),\text{ for }\zeta\in \partial U\}.
\]
For large~$n$, we have
\[
h_n(z)\ge
\left\{
 \begin{array}{ll}
 \underline{a}_{\,n},&\text{for }z\in \partial U,\\
 b_n,&\text{for }z\in\partial_{\infty}U.
 \end{array}
 \right.
\]
So, if $u\in {\mathcal P}$, then, by the maximum principle \cite[Theorem~2.3.1]{tR95},
\[
\frac{h_n(z)-\underline{a}_{\,n}}{b_n-\underline{a}_{\,n}}\ge u(z),\;\;\text{for
}z\in U,
\]
and hence, by the definition of $\omega(z,\partial_{\infty}U,U)$ above,
\begin{equation}\label{hn-lower}
\frac{h_n(z)-\underline{a}_{\,n}}{b_n-\underline{a}_{\,n}}\ge  \omega(z,\partial_{\infty}U,U),\;\;\text{for
}z\in U.
\end{equation}

Suppose now that $b=\infty$. Then we deduce from \eqref{hn-lower} on letting $n\to\infty$ that $\omega(z,\partial_{\infty}U,U)=0$, for $z\in U$, which proves part~(d)(ii).

Suppose next that $b<\infty$. Then we deduce from \eqref{hn-lower} and Theorem~\ref{main2c} that
\begin{equation}\label{h-lower}
h(z)\ge a+(b-a)\omega(z,\partial_{\infty}U,U),\;\;\text{for }z\in
U.
\end{equation}
On the other hand, by \eqref{ha} and the fact that $b=\sup_{z\in U}h(z)$,  we have
\[
\limsup_{z\to\zeta} \frac{h(z)-a}{b-a}\le \chi^{_{_{}}}_{\partial_{\infty}U}(\zeta),\;\;\text{for }\zeta\in\partial U.
\]
Thus we deduce, by the definition of $\omega(z,\partial_{\infty}U,U)$, that
\[
\frac{h(z)-a}{b-a}\le \omega(z,\partial_{\infty}U,U),\;\;\text{for }z\in U,
\]
which, together with \eqref{h-lower}, proves part~(d)(i).
\end{proof}

Next we consider the level sets of the harmonic function $h$ defined in~\eqref{hdef} and give a precise description of the location of the images of these level sets. For each $n \in \N$, we introduce the function
\begin{equation}\label{hntilde}
\tilde{h}_n(z)=\lim_{m\to \infty} \frac{\log |f^m(z)|}{\log |f^m(f^n(z_0))|},\;\;z\in U_n.
\end{equation}
Then $\tilde h_n$ is a positive harmonic function in $U_n$, since $\tilde h_n(f^n(z))=h(z)$ for $z\in U$, and the level sets of $h$ in $U$ correspond to those of $\tilde h_n$ in $U_n$ under the function~$f^n$.

The following result describes the location of certain level curves of $\tilde h_n$.

\begin{theorem}\label{curves}
Let $f$ be a {\tef} with a {\mconn} {\wand} $U$, let $h$ be the harmonic function defined in \eqref{hdef}, let $a_n$, $b_n$, $\delta_n$ and $B_n$ be defined as in \eqref{Bndef} and~\eqref{Cn}, let $a$ and $b$ be defined as in Theorem~\ref{main2c}, and let
\[
\Gamma_{n,l} = f^n(\{z \in U: h(z) = l\}),\;\;\mbox{for } l \in (a,b).
\]
Then, for large $n \in \N$ and $l \in (a_n/(1 - \delta_n),
b_n(1 - 3\pi \delta_n))$, the set $\Gamma_{n,l}$ has a component
$\gamma_{n,l}$ which is a Jordan curve surrounding $0$ with
\begin{equation}\label{levelann}
\gamma_{n,l} \subset
\left\{ \begin{array}{ll} \overline{A}(r_n^{l(1-3\eps_n/(1-a))},
r_n^{l(1 + 2\delta_n)}), & \mbox{if } l \geq 1, \\
\overline{A}(r_n^{l(1-\delta_n)},r_n^{l + 3\eps_n}), & \mbox{if } l \leq 1,
\end{array}\right.
\end{equation}
where $\eps_n=\max\{\delta_n,a_n-a\}$.

Further, for $m \in \N$ and large $n\in\N$, the image set $f^m(\gamma_{n,l})$ is also a Jordan curve surrounding~$0$ and
\[
f^m(\gamma_{n,l}) \subset B_{n+m},\;\;\text{for }l \in ((a_n + 2\pi \delta_n)/(1 - \delta_n), b_n(1 - 3\pi \delta_n)/(1 + 2\delta_n)).
\]
\end{theorem}
\begin{proof}
We first note that we have not proved that $a_n \geq a$, for $n \in \N$, so we do not assume this in the proof.

We show that, if $n$ is sufficiently large and $l \in (a_n/(1 - \delta_n), b_n(1 - 3 \pi \delta_n))$, then $\Gamma_{n,l}$ contains a component $\gamma_{n,l}$ which is a Jordan curve surrounding $0$. First consider the case that $1 \leq l < b_n(1 - 3\pi \delta_n)$. In this case, $1 < l/(1 - \delta_n) < b_n - 2 \pi \delta_n$. It follows from Theorem~\ref{main2b} part~(b), Theorem~\ref{convex} and Theorem~\ref{cover} part (a) that, for $m \in \N$, large $n \in \N$ and $c \in (l/(1-\delta_n),b_n - 2\pi \delta_n)$,
\begin{eqnarray*}
m(r_n^{c},f^m)& \geq & M(r_n^c,f^m)^{1 - \delta_n} > M(r_n^{l/(1 - \delta_n)},f^m)^{1-\delta_n}\\
& \geq & M(r_n,f^m)^l \geq r_{n+m}^l.
\end{eqnarray*}
Thus, for large $n\in\N$, $z \in U$ and $1 \leq l < b_n(1 - 3\pi \delta_n)$,
\begin{equation}\label{hb1}
\mbox{ if } |f^n(z)| = r_n^{c}, \mbox{ where } c \in (l/(1-\delta_n),b_n - 2\pi \delta_n), \mbox{ then } h(z) \ge l.
\end{equation}

Next, consider the case that $a_n/(1 - \delta_n) <l \le 1$. In this case, $a_n < l(1 - \delta_n) <1$. It follows from Theorem~\ref{convex} and Theorem~\ref{cover} part (a) that, for $m\in\N$, large $n \in \N$ and $c \in (a_n, l(1 - \delta_n))$,
\[
M(r_n^c,f^m) < M(r_n^{l(1-\delta_n)},f^m) \leq M(r_n,f^m)^{l(1-\delta_n)} \leq r_{n+m}^l.
\]
Thus, for large $n\in\N$, $z \in U$ and $a_n/(1-\delta_n)<l \le 1$,
\begin{equation}\label{ha1}
\mbox{ if } |f^n(z)| = r_n^{c}, \mbox{ where } c \in (a_n, l(1 - \delta_n)),
\mbox{ then } h(z) \le l.
\end{equation}

It follows from \eqref{hb1} and~\eqref{ha1} together with the maximum principle that, if~$n$ is sufficiently large, then, for $l \in (a_n/(1-\delta_n), b_n(1 - 3\pi \delta_n))$, the set $\gamma_{n,l}$ can be taken to be the outer boundary of the component of the set $f^n(\{z \in U: h(z) < l\})$ that surrounds~$0$, and is therefore a Jordan curve.

The statement \eqref{hb1} implies that for large $n\in\N$ the upper bound for $\gamma_{n,l}$ in \eqref{levelann} holds in the case $l\ge 1$ and \eqref{ha1} implies that for large $n\in\N$ the lower bound for $\gamma_{n,l}$ holds in the case $l\le 1$.

To obtain the other bounds for $\gamma_{n,l}$ in \eqref{levelann}, we use the function $\tilde h_n$ defined in~\eqref{hntilde}. Note that
$\gamma_{n,l}$ is the unique component of $\{z\in U_n:\tilde h_n(z)=l\}$ that surrounds~$0$, by the maximum principle.

By \eqref{hb1}, for sufficiently large $n\in\N$ and $1 \leq l < b_n(1 - 3\pi \delta_n)$,
\begin{equation}\label{hnb1}
\mbox{ if } |z| = r_n^{c}, \mbox{ where } c \in (l/(1-\delta_n),b_n - 2\pi \delta_n), \mbox{ then } \tilde h_n(z) \ge l,
\end{equation}
and, by \eqref{ha1}, for sufficiently large $n\in\N$ and $a_n/(1 - \delta_n) <l \le 1$,
\begin{equation}\label{hna1}
\mbox{ if } |z| = r_n^{c}, \mbox{ where } c \in (a_n, l(1 - \delta_n)),
\mbox{ then } \tilde h_n(z) \le l.
\end{equation}
From now on we assume in this proof that $n$ is sufficiently large for \eqref{hnb1} and \eqref{hna1} to hold.

By Theorem~\ref{h-measure} part~(a),
\begin{equation}\label{ahb}
a\le\tilde h_n(z)\le b, \;\text{ for } z\in U_n,
\end{equation}
and, by \eqref{hnb1} with $l=1$,
\[
\tilde h_n(z)\ge 1,\;\text{ for }|z|= r_n^{1 +2\delta_n}.
\]
Thus, by the maximum principle, in the annulus $A(r_n^{a_n},r_n^{1 +2\delta_n})$ the function $\tilde h_n$ majorises  the log-linear function $\underline L(r)$ that satisfies
\[
\underline L(r_n^{a_n})=a\quad\text{and}\quad \underline L(r_n^{1 +2\delta_n})=1.
\]
So, by a calculation, for $z\in A(r_n^{a_n},r_n^{1 +2\delta_n})$,
\begin{equation}\label{L0}
\tilde h_n(z)\ge \underline L(|z|)=\left(\frac{1-a}{1+2\delta_n-a_n}\right)\frac{\log |z|}{\log r_n}+\frac{a(1+2\delta_n)-a_n}{1+2\delta_n-a_n}.
\end{equation}
Thus, for $a_n<l\le 1$ and $|z|=r_n^{l+\eps}$, where $l+\eps<1+2\delta_n$, we have
\[
\tilde h_n(z)\ge \underline L(r_n^{l+\eps})=\left(\frac{1-a}{1+2\delta_n-a_n}\right)(l+\eps)+\frac{a(1+2\delta_n)-a_n}{1+2\delta_n-a_n}\ge l,
\]
provided that
\[
\eps\ge (a_n-a)\left(\frac{1-l}{1-a}\right)+2\delta_n\left(\frac{l-a}{1-a}\right).
\]
In particular, for $a_n<l\le 1$, we deduce that
\[
\mbox{ if } |z| = r_n^{c}, \mbox{ where } c \in (l+\max\{a_n-a,0\} +2\delta_n,1+2\delta_n),
\mbox{ then } \tilde h_n(z) \ge l.
\]
This gives the upper bound for $\gamma_{n,l}$ in \eqref{levelann} in the case $l\le 1$.


Now let $I(r)$, $r\in(r_n^{a_n},r_n^{b_n})$, denote the mean value of $\tilde h_n$ on $\{z:|z|=r\}$. Then $I(r)$ is log-linear, since $\tilde h_n$ does not vanish in $U_n$. Also, by~\eqref{hna1} with $l=1$ and~\eqref{ahb},
\begin{equation}\label{I1}
I(r) \ge a,\;\text{for }r_n^{a_n}<r<r_n^{b_n},\;\;\text{and}\;\; I(r_n^{1 -\delta_n}) \le 1.
\end{equation}
Further,
\begin{equation}\label{I3}
(1-\delta_n)\tilde h_n(z)\le I(|z|),\;\text{ for }z\in \overline{A}(r_n^{a_n+2\pi\delta_n},r_n^{b_n-2\pi\delta_n}),
\end{equation}
by an application of Harnack's inequality similar to that in the proof of Theorem~\ref{Har}.

It follows from \eqref{I1} that
\begin{equation}\label{I4}
I(r)\le \overline L(r),\;\text{ for }r_n^{1-\delta_n}<r<r_n^{b_n},
\end{equation}
where $\overline L(r)$ is the log-linear function that satisfies
\[
\overline L(r_n^{a_n})=a\quad\text{and}\quad  \overline L(r_n^{1-\delta_n})=1;
\]
that is,
\[
\overline L(r)=\left(\frac{1-a}{1-\delta_n-a_n}\right)\frac{\log r}{\log r_n}+\frac{a(1-\delta_n)-a_n}{1-\delta_n-a_n}.
\]
By \eqref{I3} and \eqref{I4},
\[
(1-\delta_n)\tilde h_n(z)\le I(|z|)\le \overline L(|z|),\;\text{ for } r_n^{1-\delta_n}<|z|<r_n^{b_n-2\pi \delta_n}.
\]
Thus, for $1\le l< b_n-2\pi\delta_n$ and $|z|=r_n^{l-\eps}$, where
 $1-\delta_n<l-\eps$,
\[
(1-\delta_n)\tilde h_n(z)\le\overline L(r_n^{l-\eps})
=\left(\frac{1-a}{1-\delta_n-a_n}\right)(l-\eps)+\frac{a(1-\delta_n)-a_n}{1-\delta_n-a_n}
\le (1-\delta_n)l,
\]
provided that
\[
\eps\ge (a_n-a)\left(\frac{l-1}{1-a}\right)+\delta_n\left(\frac{2l-a-\delta_n l-a_nl}{1-a}\right).
\]
In particular, for $1\le l< b_n-2\pi\delta_n$, we deduce that
\[
\mbox{ if } |z| = r_n^{c}, \mbox{ where } c \in (1-\delta_n,l-l(\max\{a_n-a,0\}+2\delta_n)/(1-a)),
\mbox{ then } \tilde h_n(z) \le l.
\]
This gives the lower bound for $\gamma_{n,l}$ in \eqref{levelann} in the
case $l\ge 1$.

Further, since $\gamma_{n,l}$ is a level curve of $\tilde h_n$, it is not difficult to see that $f^m(\gamma_{n,l})$, $m\in\N$, must be a component of a level set of $\tilde h_{n+m}$ and so must also be a Jordan curve.

Finally, if $n$ is sufficiently large, then it follows from \eqref{levelann} that $\gamma_{n,l} \subset C_n$, for $l \in ((a_n + 2\pi\delta_n)/(1-\delta_n), b_n(1-3\pi\delta_n)/(1+2 \delta_n))$. Hence, for each $m \in \N$, we have $f^m(\gamma_{n,l}) \subset B_{n+m}$, by Theorem~\ref{cover} part~(c).
\end{proof}

\section{Connectivity properties of multiply connected wandering domains}\label{conn}
\setcounter{equation}{0}
We first prove Theorem~\ref{main3} which concerns the relationship between the connectivity of a multiply connected wandering domain and the location of the critical points.

\begin{proof}[Proof of Theorem~\ref{main3}]
Let $U$ be a multiply connected {\wand} and let $U_n=f^n(U)$
for $n\geq 0$. We need only prove part~(a) since parts~(b) and~(c) follow from part~(a) together with Theorem~B.

Suppose first that $\bigcup_{m=n}^\infty U_m$ contains no
critical points. We want  to show that $U_n$ is doubly connected.
Suppose that this is not the case. Then there exists a
compact, connected subset $C$ of $U_n$ such that
no doubly connected domain $A$ satisfies $C\subset A\subset U_n$.
By Theorem~\ref{main2a} we have
$f^{m-n}(C)\subset C_{m}$ for sufficiently large $m$, with the annulus
$C_{m}$ defined as in~\eqref{Cn}. For such an~$m$, let $A$ be the component of $f^{-(m-n)}(C_{m})$
that contains~$C$. Since $A$ contains no critical
points of $f^{m-n}$, it follows from the Riemann-Hurwitz formula \eqref{RieHur} that $A$ is doubly connected,
which is a contradiction. Thus we have proved that $U_n$ is doubly connected. This proves one direction of part~(a).

As mentioned in the introduction, the other direction of part~(a) was given already by Kisaka and Shishikura in Theorem~B. Indeed, if $c(U_n)=2$, then $c(U_{n+1})=2$ and $U_n$ contains no critical points of $f$, by the Riemann-Hurwitz formula. Thus the other direction of part~(a) follows.
\end{proof}

We now show that the inner and outer connectivities of a multiply connected wandering domain $U$ have similar properties to the connectivity of $U$ with one interesting exception described in part (c) of the following result.

First recall from~\eqref{Cn} and~\eqref{oi} that, if $z_0 \in U$ and $r_n = |f^n(z_0)|$, then, for large $n \in \N$,
\[
C_n = A\left(r_n^{a_n+2\pi\delta_n},r_n^{b_n(1-3\pi\delta_n)}\right),
\]
\[
U_n^{+} = U_n \cap \{z: |z| > r_n\}\;\;\text{and}\;\;U_n^{-} = U_n \cap \{z: |z| < r_n\}.
\]

\begin{theorem}\label{main4a}
Let $f$ be a {\tef} with a {\mconn} {\wand} $U$ and let $N\in \N$ be such that $f^m$ has no critical points in $C_n$, for $n\ge N$ and $m\in \N$ (which is the case for sufficiently large $N$ by Theorem~\ref{nocrits}). For $n\ge N$,
\begin{itemize}
\item[(a)] Theorem B and Theorem~\ref{main3} remain valid with $U_n$ and $U_m$ replaced by $U_n^-$ and $U_m^-$, respectively;
\item[(b)] if $c(U_N^-) < \infty$, then Theorem B and Theorem~\ref{main3} remain valid with $U_n$ and $U_m$ replaced by $U_n^+$ and $U_m^+$, respectively;

\item[(c)] if $c(U_N^-) = \infty$, then one of the following holds:
\begin{itemize}
\item[(i)] $c(U_n^+) = 2$, for $n \geq N$, and $\bigcup_{n=N}^{\infty}U_n^+$ contains no critical points of~$f$;
\item[(ii)] there exists $N_0 \geq N$ such that $c(U_n^+) = \infty$, for $N \leq n \leq N_0$, $c(U_n^+) = 2$, for $n > N_0$, and  $\bigcup_{n=N}^{N_0} U_n^+$ contains a finite non-zero number of critical points of $f$;
\item[(iii)] $c(U_n^+) = \infty$, for $n \geq N$, and $\bigcup_{n=N}^{\infty}U_n^+$ contains infinitely many critical points of $f$.
\end{itemize}
\end{itemize}
\end{theorem}
\begin{proof}
We begin by proving part~(a). We claim that, if $n$ is sufficiently large, then, for each $m \geq n$, there exists a unique set $U_{n,m}^- \subset U_n\cap \{z: |z| < r_n^{1 + 2\delta_n}\}$ such that $f^{m-n}$ is a proper map of $U_{n,m}^-$ onto $U_m^{-}$ and, for $n \leq p \leq m$,
\begin{equation}\label{Umn}
f^{p-n}(U_{n,m}^-) \subset U_p^- \cup C_p.
\end{equation}
This follows since, for large $n$ and $p \geq n$, we have
\begin{align}
f^{p-n}(\{z:|z|=r_n^{1+2\delta_n}\})&\subset \{z:|z|\ge
M(r_n^{1+2\delta_n},f^{p-n})^{(1-\delta_n)}\}\notag\\
&\subset \{z: |z| \geq M(r_n,f^{p-n})\}\notag\\
&\subset \{z:|z|\ge r_p\},
\end{align}
by Theorem~\ref{main2b} part~(b), Theorem~\ref{convex} and Theorem~\ref{cover} part~(a). Hence there exists a component $U_{n,m}^-$ of $f^{-(m-n)}(U_m^{-})$ in $U_n \cap \{z: |z| < r_n^{1 + 2\delta_n}\}$ and $f^{m-n}$ is a proper map of $U_{n,m}^-$ onto $U_m^{-}$. If $n$ is sufficiently large, then $U_{n,m}^-$ also satisfies~\eqref{Umn} since, for $n \leq p \leq m$,
\[
 M(r_n^{1+2\delta_n},f^{p-n}) < r_p^{b_p(1-3\pi\delta_p)},
\]
by Lemma~\ref{delta} part~(d).

The results of part~(a) that are analogous to Theorem B now follow from the Riemann-Hurwitz formula applied to the map $f^{m-n}:U_{n,m}^- \to U_m^{-}$; that is,
\begin{equation}\label{RH}
c(U_{n,m}^-)-2 = d_{n,m} (c(U_m^-)-2)+N_{n,m},
\end{equation}
where $d_{n,m}$ is the degree of the map and $N_{n,m}$ is the number of critical points of $f^{m-n}$ in $U_{n,m}^-$. We note that the sets $U_n^-$ and $U_{n,m}^-$ have equal connectivity and also contain equal numbers of critical points of $f$, by Theorem~\ref{nocrits}, and $d_{n,m}>1$ for large $n$ and $m>n$, by Picard's theorem. We deduce from \eqref{RH}, by considering the case when $m-n=1$, that the sequence $(c(U_m^-))$ is non-increasing. Also, if $c(U_n^-)<\infty$ for some $n$, then for sufficiently large $m$ we have $c(U_m^-)=2$ and so~$f$ has no critical points in $U_m^-$, for such $m$, by \eqref{RH} again, as required.
%
%

To complete the proof of part~(a), we suppose that $\bigcup_{m=n}^{\infty}U_m^-$ contains no critical points of $f$. We want to show that $U_n^-$ is doubly connected. If this is not the case, then there exists a
compact, connected subset $C$ of $U_n^- \cap \{z: |z| \leq r_n^{a_n}\}$ such that
no doubly connected domain $A$ satisfies $C\subset A\subset U_n^- \cup B_n$.
By Theorem~\ref{main2a} we have
$f^{m-n}(C)\subset C_{m}$ for large~$m$. Also, by Theorem~\ref{cover} part~(a) and Theorem~\ref{convex}, if $n$ is sufficiently large, then $f^{m-n}(C) \subset U_m^{-}$ for large~$m$.  Let~$A$ be the component of $f^{-(m-n)}(B_{m} \cap U_m^{-})$
that contains~$C$. Clearly $A \subset U_{n,m}^-$ and so, by~\eqref{Umn} and Theorem~\ref{nocrits}, $A$ contains no critical
points of $f^m$. It then follows from the Riemann-Hurwitz formula that $A$ is doubly connected, which is a contradiction. Thus we have proved that
$U_n^-$ is doubly connected. This completes the proof of part~(a).

We now prove parts~(b) and~(c). We argue as in part~(a) but consider the set $U_m^{+}$ instead of $U_m^{-}$. It follows from Theorem~\ref{cover} part~(a) and Theorem~\ref{convex} that, for large $n$ and $m\geq n$, there is a unique pre-image component $U_{n,m}^+ \subset U_n \cap \{z:|z| \geq r_n^{1-\delta_n}\}$ such that $f^{m-n}$ is a proper map of $U_{n,m}^+$ onto $U_m^{+}$ and, for $n \leq p \leq m$,
\begin{equation}\label{Umnp}
f^{p-n}(U_{n,m}^+) \subset U_p^+ \cup C_p.
\end{equation}

We can apply the Riemann-Hurwitz formula to the sets $U_{n,m}^+$ and $U_m^{+}$ but need to make some adjustments to obtain the final results since the connectivity of $U_{n,m}^+$ may be less than the connectivity of $U_n^+$. The difference occurs when $c(U_{n}^+)>2$, in which case there exists $m>n$ such that, under $f^{m-n}$, at least two of the bounded complementary components of $U_{n,m}^{+}$ map onto $\{z: |z| < r_m\}$; this is the case because the backwards orbit of~$0$ is dense in $J(f)$. Let $H_{n,m}$ denote a component of the interior of one such complementary component, with the property that $H_{n,m} \subset \{z: |z| > r_n\}$. Then
\[
c(U_n^+) \geq c(U_{n,m}^+) + c(H_{n,m} \cap U_n) - 2.
\]
Applying the Riemann-Hurwitz formula to each of the proper maps
\[
f^{m-n}: U_{n,m}^+ \to U_m^+ \mbox{ and } f^{m-n} : H_{n,m} \cap U_n \to U_m^-,
\]
 we see that
\[
c(U_{n,m}^+) \geq c(U_m^+)
\]
 and
 \[
 c(H_{n,m} \cap U_n) \geq c(U_m^-).
 \]
 Putting these inequalities together, gives
 \[
 c(U_n^+) \geq c(U_m^{+}) + c(U_m^-)-2.
  \]
  Thus, if $c(U_N^-)$ and hence $c(U_m^-)$ is infinite, then $c(U_n^+)$ is also infinite. The remaining arguments go through as before.
\end{proof}

We conclude this section by using Theorem~\ref{h-measure} to prove Theorem~\ref{main5}.

\begin{proof}[Proof of Theorem~\ref{main5}]
Since $U$ has eventual outer connectivity~$2$, we can assume that $U$ itself has outer connectivity~$2$. In this situation the outer boundary component of $U$ is isolated from the rest of $\partial U$, so the harmonic measure of the outer boundary component of $U$ with respect to $U$ is positive, by \cite[Theorem~4.3.4(b)]{tR95}, for example. Thus $b<\infty$, by Theorem~\ref{h-measure}.
\end{proof}

\section{Proof of Theorem~\ref{stable}}
In this section we use our results to prove Theorem~\ref{stable}, which states that the property of having a multiply connected wandering domain is stable under relatively small changes to the function.

Let $f,g$ be transcendental entire functions and suppose that $f$ has a multiply connected {\wand} $U$ and that there exists $\alpha \in (0,1)$ such that
\begin{equation}\label{fg}
 M(r,g-f) \leq  M(r,f)^{\alpha}, \;\;\mbox{for large } r.
\end{equation}
We show that $g$ also has a multiply connected {\wand}.

First, let $z_0 \in U$ and, for large $n \in \N$, let $r_n$, $a_n$, $b_n$ and $B_n$ be defined as in~\eqref{Bndef}, and let~$b$ be defined as in Theorem~\ref{main2c}. If $b = \infty$, then we take $b'$ to be some value in $(1, \infty)$ and if $b < \infty$, then we let $b' = b$. For $n>N$, we define
\[
\eps_n = \frac{8\pi}{(b'-1)\log r_n},\;c_n = \left(1 + \frac{b'-1}{2}\right) \prod_{m=N}^{n-1}(1 - 2 \eps_m) \mbox{ and }  d_n =  \beta b_n \prod_{m=N}^{n-1}(1 + 2K \eps_m).
\]
Here $\beta \in(0,1)$ is chosen below and $K>1$ is the constant from Theorem~\ref{cover}. We claim that, if $N$ is sufficiently large, then, for $n \geq N$,
\begin{align}\label{gA}
g(A(r_n^{1 + (b'-1)/2},r_n^{1 + 5(b'-1)/8})) &\subset g(A(r_n^{c_n},r_n^{d_n}))\\
&\subset A(r_{n+1}^{c_{n+1}},r_{n+1}^{d_{n+1}}) \subset A(r_n,r_n^{b_n}).\notag
\end{align}
It follows from~\eqref{gA} and Theorem~\ref{main2c} part (b)(i) that, if $N$ is sufficiently large, then for $n \geq N$ and $m \in \N$,
\[
g^m(A(r_n^{1+(b'-1)/2},r_n^{
1+5(b'-1)/8})) \subset A(r_{n+m},r_{n+m}^b)
 \]
 and so, by Montel's theorem, $A(r_n^{1 +(b'-1)/2},r_n^{1+5(b'-1)/8}) \subset F(g)$, for $n \geq N$. Thus~$g$ has a multiply connected wandering domain.

We now show that~\eqref{gA} does indeed hold.  We begin by noting that, since $z_0$ belongs to a multiply connected {\wand} and $r_n = |f^n(z_0)|$, it follows from~\cite[Theorem 1]{RS00} that
 \[
 \frac{1}{n} \log \log r_n \to \infty \;\mbox{ as } n \to \infty.
 \]
 Thus, by the definition of $\eps_n$,
 \[
 \lim_{N \to \infty} \prod_{m=N}^{\infty} (1 - 2 \eps_m)
 =  \lim_{N \to \infty} \prod_{m=N}^{\infty} (1 + 2K \eps_m) = 1.
 \]

Together with Theorem~\ref{main2c} part (a), this implies that, by choosing $N$ sufficiently large and $\beta$ appropriately and considering the cases that $b= \infty$ and $b < \infty$ separately, we can ensure that, for $n \geq N$,
\begin{equation}\label{cond1}
a_n + (b'-1)/4 < 1 + (b'-1)/4 < c_n < 1 + (b'-1)/2,
\end{equation}
\begin{equation}\label{cond3}
1 + 5(b'-1)/8 < d_n < b_n - (b'-1)/4 < b_n
 \end{equation}
 and
 \begin{equation}\label{cond2}
 (b'-1)/4 \in (\pi/\log r_n, (b_n - a_n)/2).
\end{equation}

If $z \in A(r_n^{c_n}, r_n^{d_n})$ for some $n \geq N$ with $N$ sufficiently large, then it follows from~\eqref{cond1},~\eqref{cond3} and~\eqref{cond2} that we can apply Theorem~\ref{main2b} part (a) with $\rho = |z|$ and $\eps = (b'-1)/4$. Together with~\eqref{fg}, Theorem~\ref{convex} and Theorem~\ref{cover} part (a), this implies that
\begin{eqnarray*}
|g(z)| & \geq & m(|z|,f) - M(|z|,g-f)\\
& \geq & M(|z|,f)^{1-\eps_n} - M(|z|,f)^{\alpha}> M(|z|,f)^{1-2\eps_n}\\
&\geq & M(r_n^{c_n},f)^{1-2\eps_n}\geq M(r_n,f)^{c_n(1-2\eps_n)}\\
&\geq&  r_{n+1}^{c_n(1-2\eps_n)} = r_{n+1}^{c_{n+1}}.
\end{eqnarray*}

Next, by Theorem~\ref{cover} part (b), for sufficiently large $n$,
\begin{equation}\label{rb}
r_{n+1}^{b_{n+1}} \geq M(r_n^{b_n}/K,f) > M(r_n^{b_n(1-K\eps_n)},f).
\end{equation}
If $z \in A(r_n^{c_n}, r_n^{d_n})$ for some $n \geq N$ with $N$ sufficiently large, then it follows from~\eqref{fg}, Theorem~\ref{convex} and~\eqref{rb} that
\begin{eqnarray*}
|g(z)| & \leq & M(r_n^{d_n},f) + M(r_n^{d_n},f)^{\alpha}\\
& < & M(r_n^{d_n},f)^{1+\eps_n}\leq M(r_n^{b_n(1 - K\eps_n)},f)^{d_n(1+\eps_n)/(b_n(1-K\eps_n))}\\
& < & r_{n+1}^{b_{n+1}d_n(1+\eps_n)/(b_n(1-K\eps_n))}\\
& = & r_{n+1}^{b_{n+1}\beta(1+\eps_n)/(1-K\eps_n)\prod_{m=N}^{n-1}(1+2K\eps_m)}\\
& < & r_{n+1}^{d_{n+1}}.
\end{eqnarray*}

We have now shown that~\eqref{gA} holds and this completes the proof.

\section{Examples}\label{examples}
\setcounter{equation}{0}
In this section we construct a number of new examples of transcendental entire functions with multiply connected wandering domains. These have properties that differ significantly from earlier examples of such functions. As mentioned in the introduction, all earlier examples have the property that the limits~$a$ and~$b$ defined in Theorem~\ref{main2c} satisfy $b/a = \infty$. We now show that, for any $c > 1$, there exists a {\tef} with a multiply connected {\wand} for which $b/a \leq c$.
\begin{example}\label{exa1}
Let $c \in (1, \infty)$ and consider the transcendental entire
function defined by
\begin{equation}\label{prod}
 f(z) = z \prod_{i=1}^{\infty} \left( 1 - \frac{z}{s_i^c}  \right) \left( 1 - \frac{z}{s_{i^2}}  \right),
\end{equation}
where $s_1>1$ and $s_{n+1} = s_n^{n + [\sqrt{n}\,]}$. Then $f$ has a multiply connected {\wand} $U$ such that the limits $a$ and $b$ defined in Theorem~\ref{main2c} satisfy $b/a \leq c$. Hence, for large $n \in \N$, $f^n(U)$ does not contain an annulus of the form $A(t,t^C)$, where $t>0$ and $C>c$.
\end{example}
\begin{proof} We will show that there exist sequences $(\alpha_n)$, $(\beta_n)$ with $1\leq \alpha_n \leq \beta_n \leq c$, for $n \in \N$, such that, for sufficiently large $n$,
\begin{equation}\label{image}
f(A_n) \subset A_{n+1}, \; \mbox{ where } A_n = A(s_n^{\alpha_n},
s_n^{\beta_n}).
\end{equation}
This is sufficient to show that there exists $N \in \N$ and a multiply connected {\wand} $U$ such that $A_N \subset U$ and $f^n(A_N) \subset A_{n+N}$, for $n \in \N$.

We then choose a sequence $(n_k)$ such that $n_k + N = k^2$ for $k$ sufficiently large. For large $k$, the points $s_{n_k+N}$ and $s_{n_k + N}^c$ are zeros of $f$ and so do not belong to $f^{n_k}(U)$.  Taking $B_n = A(r_n^{a_n},r_n^{b_n})$ to be the annulus defined in~\eqref{Bndef}, it follows from~\eqref{image} and Theorem~\ref{main2a} that, for large $k$, we have $B_{n_k} \subset A(s_{n_k},s_{n_k + N}^c)$ and hence $b/a = \lim_{n \to \infty} b_n/a_n \leq c$. It follows from Theorem~\ref{main2c} part (b) that, for large $n$, $B_n$ is the largest annulus contained in $f^n(U)$ and so $f^{n}(U)$ does not contain an annulus of the form $A(t,t^C)$, where $t>0$ and $C>c$.

In order to show that~\eqref{image} holds, we prove the following result.
\begin{lemma}\label{ex}
Let $f$ be the {\tef} defined in~\eqref{prod}. Then there exists
$N_1 \in \N$ such that, if $n \geq N_1$ and $c^{1/8} \leq \alpha \leq
\beta \leq c^{7/8}$, then
\begin{equation}\label{notsq}
 f\left( A \left( s_n^\alpha, s_n^\beta\right) \right) \subset A \left( s_{n+1}^{\alpha(1-2c/n^2)},s_{n+1}^{\beta(1+2c/n^2)} \right),
\end{equation}
for $m^2 < n < (m+1)^2$, $m \in \N$, and
\begin{equation}\label{sq}
  f\left( A \left( s_n^\alpha, s_n^\beta\right) \right) \subset A \left( s_{n+1}^{\alpha(1-2/n)},s_{n+1}^{\beta(1+2/n)} \right),
\end{equation}
for $n = m^2$,  $m \in \N$.
\end{lemma}
\begin{proof}
First suppose that $m^2 < n < (m+1)^2$, for some $m \in \N$, and
that $|z| = s_n^k$, where $c^{1/8} \leq k \leq c^{7/8}$. Then
\begin{equation}\label{size}
 |f(z)| = |z|\left| \prod_{i=n}^{\infty}\left( 1 - \frac{z}{s_i^c}\right) \prod_{i=m+1}^{\infty}\left( 1 - \frac{z}{s_{i^2}}\right)\right|
 \left| \prod_{i=1}^{n-1}\frac{z}{s_i^c}\left(1 - \frac{s_i^c}{z} \right) \prod_{i=1}^{m}\frac{z}{s_{i^2}}\left(1 - \frac{s_{i^2}}{z} \right) \right|.
\end{equation}
Since $s_{n+1} \geq s_n^{n+1}$, it follows from~\eqref{size}
that, if $n$ is sufficiently large, then
\begin{eqnarray*}
|f(z)| & \geq & \frac{1}{2} |z| |z^{n-1}| |z^m|
\prod_{i=1}^{n-1}\frac{1}{s_i^c}
\prod_{i=1}^m \frac{1}{s_{i^2}}\\
& \geq & \frac{1}{2} |z|^{n+m}\frac{1}{s_{n-1}^{c(1+2/n)}} \frac{1}{s_{n-1}^{1+2/n}}\\
& \geq & \frac{1}{2} |z|^{n+m-2c/n}\geq  |z|^{(n+m)(1-2c/n^2)}\\
& = &  s_n^{k(n+m)(1-2c/n^2)} = s_{n+1}^{k(1-2c/n^2)}.
\end{eqnarray*}
Similarly, $|f(z)| \leq s_{n+1}^{k(1+2c/n^2)}$, if $n$ is
sufficiently large.

Now suppose that $n = m^2$, for some $m \in \N$, and that $|z| =
s_n^k$, where $c^{1/8} \leq k \leq c^{7/8}$. Since $s_{n+1} \geq
s_n^{n+1}$, it follows from~\eqref{size} that, if $n$ is
sufficiently large, then
\begin{eqnarray*}
|f(z)| & \geq & \frac{1}{2} |z| |z^{n-1}| |z|^{m}
\prod_{i=1}^{n-1}\frac{1}{s_i^c}
\prod_{i=1}^m \frac{1}{s_{i^2}}\\
& \geq & \frac{1}{2} |z|^{n+m}\frac{1}{s_{n-1}^{c(1+2/n)}} \frac{1}{s_{n}^{1+2/n}}\\
& \geq & \frac{1}{2} |z|^{n+m-2}\geq  |z|^{(n+m)(1-2/n)}\\
& = &  s_n^{k(n+m)(1-2/n)}=  s_{n+1}^{k(1-2/n)}.
\end{eqnarray*}
Similarly, $|f(z)| \leq s_{n+1}^{k(1+2/n)}$, if $n$ is
sufficiently large. This completes the proof.
\end{proof}

We now take $N_1 \in \N$ to satisfy the conditions of
Lemma~\ref{ex} and take $N_2 \geq N_1$ such that
\[
 \prod_{n=N_2}^{\infty}(1 - 2c/n^2) > 1/c^{1/16} \mbox{ and } \prod_{n=N_2}^{\infty}(1+2c/n^2) < c^{1/16}.
\]

Then, if we take
\[
 \alpha_n = c^{1/4}\prod_{i=N_2}^{n-1}(1 - 2c/i^2)\prod_{i=N_2}^{[\sqrt{n}\,]}(1 - 2/i^2)
\]
and
\[
\beta_n = c^{3/4}\prod_{i=N_2}^{n-1}(1 +
2c/i^2)\prod_{i=N_2}^{[\sqrt{n}\,]}(1 + 2/i^2),
\]
we have, for $n \geq N_2^2$,
\[
 A_n = A(s_n^{\alpha_n}, s_n^{\beta_n}) \subset A(s_n^{c^{1/8}}, s_n^{c^{7/8}})
\]
and so, by Lemma~\ref{ex},
\[
 f(A_n) \subset A_{n+1}.
\]
Thus~\eqref{image} holds. This completes the proof of Example~\ref{exa1}.
\end{proof}

We have already noted that, if a {\tef} $f$ has a multiply connected {\wand} $U$ and $B_n = A(r_n^{a_n},r_n^{b_n}) \subset f^n(U)$ is as defined in~\eqref{Bndef}, then for all the examples in earlier papers, $b/a = \lim_{n \to \infty} b_n/a_n = \infty$. In fact these examples also appear to have the property that $a = 0$ and $b < \infty$.

We now give an example of a multiply connected {\wand} $U$ for which $b = \infty$. Moreover, in contrast to the previous example, this example has the property that, for some large values of $n \in \N$, the {\wand} $U_n$ contains a very large annulus that is disjoint from the main annulus $B_n$ defined in~\eqref{Bndef}.

\begin{example}\label{exa2}
Consider the transcendental entire function defined by
\begin{equation}\label{prod2}
 f(z) = z^2 \prod_{n=1}^{\infty}  \prod_{m=(n!)^2}^{((n+1)!)^2-1} \left( 1 - \frac{z}{s_m^{n!}}  \right),
\end{equation}
where $s_1 > 1$ and $s_{m+1} = s_m^{m+1}$, for $m \in \N$. Then $f$ has a multiply connected {\wand} $U$ such that, for large $n \in \N$, $U_{(n!)^2-1}$ contains an annulus of the form $A(R,R^{n-1})$ that is disjoint from the annulus $B_{(n!)^2-1}$ defined in~\eqref{Bndef}, with $R > r_{(n!)^2-1}^{b_{(n!)^2-1}}$. Moreover, the limit $b$ defined in Theorem~\ref{main2c} is infinite.
\end{example}

We will show that, for sufficiently large $n$, there exist
sequences $(a_{m,n})$, $(b_{m,n})$ with $1 \leq a_{m,n} \leq 2$
and $b_{m,n} = n!-1$, for $m \geq (n!)^2$, such that
\begin{equation}\label{image2}
f(A_{m,n}) \subset A_{m+1,n}, \; \mbox{ where } A_{m,n} =
A(s_m^{a_{m,n}}, s_m^{b_{m,n}}),\;\;\mbox{for } m \geq (n!)^2.
\end{equation}
This is sufficient to show that there exists a multiply connected {\wand} $U$ such that, for sufficiently large $n$ and $m \geq (n!)^2$,  $A_{m,n} \subset U_m = f^m(U)$.

 We also show that, if $n$ is sufficiently large, and
 \[
 A_n = A(s_{(n!)^2-1}^3,s_{(n!)^2-1}^{(n!)^2-3})
 \]
  then
  \begin{equation}\label{union}
f\left(  A_n \cap \{z: \left| 1 -
z/s_{(n!)^2-1}^{(n-1)!}\right| > 1/2 \} \right) \subset
A_{(n!)^2,n} \subset U_{(n!)^2}.
  \end{equation}
  Thus
  \[
  A_n' = A(s_{(n!)^2-1}^{((n-1)!)^2+2}, s_{(n!)^2-1}^{(n!)^2-3})
  \subset U_{(n!)^2-1}.
  \]
  Note that
  $\frac{(n!)^2 - 3}{((n-1)!)^2+2}\geq n-1$ if $n$ is sufficiently large. Also $A_n'$ and
  $A_{n!-1,n-1}$ are disjoint annuli in $U_{(n!)^2 - 1}$ with a zero at $s_{(n!)^2-1}^{(n-1)!}$ in between them. Thus $U_{(n!)^2 - 1}$ contains an annulus of the form $A(R,R^{n-1})$ that does not meet the largest annulus in $U_{(n!)^2 - 1}$ that contains $A_{n!-1,n-1}$. Moreover, it follows from~\eqref{image2} that, for large $n$ and $m \geq (n!)^2$, the annulus $B_m = A(r_m^{a_m},r_m^{b_m})$ defined in~\eqref{Bndef} must contain the annulus $A_{m,n}$. Thus, for large $n$,
  \[
   B_{(n!)^2-1} \supset A_{(n!)^2-1,n-1} \supset A\left(s_{(n!)^2-1}^2,s_{(n!)^2-1}^{(n!)^2-1}\right).
  \]
  Since there is a zero at $s_{(n!)^2-1}^{(n-1)!}$, we must have $r_{(n!)^2-1}^{b_{(n!)^2-1}} < s_{(n!)^2-1}^{(n-1)!}$ and so
  \[
  r_{(n!)^2}^{b_{(n!)^2-1}} \leq M(r_{(n!)^2-1}^{b_{(n!)^2-1}},f) < M(s_{(n!)^2-1}^{(n-1)!},f).
  \]
  We will show (in Lemma~\ref{ex3}) that
  \begin{equation}\label{bound}
  M(s_{(n!)^2-1}^{(n-1)!},f) < s_{(n!)^2}^{(n-1)!+2}
  \end{equation}
  and so
  \begin{equation}\label{upper}
  r_{(n!)^2}^{b_{(n!)^2-1}} < s_{(n!)^2}^{(n-1)!+2}.
  \end{equation}
   Also, for large $n$,
  \[
  B_{(n!)^2} \supset A_{(n!)^2,n} \supset A(s_{(n!)^2}^2,s_{(n!)^2}^{n!-1})
  \]
  and so
  \[
  r_{(n!)^2}^{b_{(n!)^2}} \geq s_{(n!)^2}^{n!-1}.
  \]
  Together with~\eqref{upper} this implies that
  \[
  b_{(n!)^2} / b_{(n!)^2-1} > (n!-1)/((n-1)!+2) \to \infty\; \mbox{ as } n \to \infty
  \]
  and so $b = \lim_{n \to \infty} b_n = \infty$.

In order to show that~\eqref{image2} holds, we prove the
following result.

\begin{lemma}\label{ex2}
Let $f$ be the {\tef} defined in~\eqref{prod2}. Then there exists
$N_1 \in \N$ such that, if $n \geq N_1$, $1 \leq a <  b \leq
n!-1$, then
\begin{equation}\label{all}
 f\left( A \left( s_m^a, s_m^b\right) \right) \subset A \left( s_{m+1}^{a - 2/m^{3/2}},s_{m+1}^{b} \right),
\end{equation}
for $m \geq (n!)^2$.
\end{lemma}

\begin{proof}
Suppose that $(N!)^2 \leq m < ((N+1)!)^2$, for some $N \in \N$,
and that $|z| = s_m^k$, where $1 \leq k \leq N!-1$. Then
\begin{equation}\label{size2}
 |f(z)| = |z|^2\left| \prod_{n=N+1}^{\infty} \prod_{i = (n!)^2}^{((n+1)!)^2-1} \left( 1 - \frac{z}{s_i^{n!}}\right)\right|
 \left| \prod_{n=1}^{N-1}\prod_{i = (n!)^2}^{((n+1)!)^2-1}\frac{z}{s_i^{n!}}\left(1 - \frac{s_i^{n!}}{z} \right) \right|
 \end{equation}
 \[
 \left|  \prod_{i=m}^{((N+1)!)^2-1} \left( 1 - \frac{z}{s_i^{N!}}\right)\right|
 \left| \prod_{i = (N!)^2}^{m-1}
\frac{z}{s_i^{N!}}\left(1 - \frac{s_i^{N!}}{z} \right) \right|.
\]

We also note that $s_m = s_1^{m!}$, so if $m \geq (n!)^2$, then
\begin{equation}\label{big}
s_m^{n!} = s_1^{m!n!} < s_1^{(m+1)!/2} = \sqrt{s_{m+1}}.
\end{equation}

It follows from~\eqref{size2} and~\eqref{big} that, if $N$ is
sufficiently large, then
\begin{eqnarray*}
|f(z)| & \geq & \frac{1}{2}  |z|^{m+1} \left(
\prod_{n=1}^{N-1}\prod_{i =
(n!)^2}^{((n+1)!)^2-1}\frac{1}{s_i^{n!}}\right) \prod_{i =
(N!)^2}^{m-1}
\frac{1}{s_i^{N!}}\\
& \geq & \frac{1}{2} |z|^{m+1}\frac{1}{s_{m-1}^{2N!}} \geq \frac{1}{2} |z|^{m+1 - 2N!/m}\\
& = & \frac{1}{2} s_m^{k(m+1 - 2N!/m)}= \frac{1}{2} s_m^{k(m+1)(1-2N!/m(m+1))}\\
& = & s_{m+1}^{k(1-2N!/m(m+1))}\geq s_{m+1}^{k(1-2/m^{3/2}))}.
\end{eqnarray*}
It also follows from~\eqref{size2} that $|f(z)| \leq |z^{m+1}| =
s_m^{k(m+1)} = s_{m+1}^k$, if $N$ is sufficiently large. This
completes the proof.
\end{proof}

It follows from Lemma~\ref{ex2} that~\eqref{image2} holds if we
take $N_2\geq N_1$ sufficiently large so that
$\sum_{i=N_2}^{\infty}2/i^{3/2} < 1$ and put, for $m \geq
(n!)^2$ and $n \geq N_2$,
\[
 a_{m,n} = 2 - \sum_{i=N_2}^{m-1}\frac{2}{i^{3/2}}
\]
and
\[
 b_{m,n} = n! - 1.
\]

We now show that~\eqref{union} and~\eqref{bound} hold if $n$ is sufficiently
large.

\begin{lemma}\label{ex3}
Let $f$ be the {\tef} defined in~\eqref{prod2}. Then there exists
$N_3 \in \N$ such that, if $N \geq N_3$ and $m = (N!)^2 - 1$
then, for $|z| = s_m^k$ with $3 \leq k \leq (N!)^2 - 3$ and
$\left|1- z/s_m^{(N-1)!}\right| > 1/2$, we have $f(z) \in
A_{m+1,N}$. Further,
\[
M(s_m^{(N-1)!},f) < s_{m+1}^{(N-1)!+2}.
\]
\end{lemma}

\begin{proof}
Suppose that $ m = (N!)^2-1$, for some $N \in \N$, and that $|z|
= s_m^k$, where $3 \leq k \leq (N!)^2 - 3$ and $\left|1-z/s_m^{(N-1)!}\right| > 1/2$. Then
\begin{equation}\label{size3}
 |f(z)| = |z|^2\left| \prod_{n=N}^{\infty} \prod_{i = (n!)^2}^{((n+1)!)^2-1} \left( 1 - \frac{z}{s_i^{n!}}\right)\right|
 \left| \prod_{n=1}^{N-2}\prod_{i = (n!)^2}^{((n+1)!)^2-1}\frac{z}{s_i^{n!}}\left(1 - \frac{s_i^{n!}}{z} \right) \right|
 \end{equation}
 \[
 \left|   \left( 1 - \frac{z}{s_m^{(N-1)!}}\right)\right|
 \left| \prod_{i = ((N-1)!)^2}^{m-1}
\frac{z}{s_i^{(N-1)!}}\left(1 - \frac{s_i^{(N-1)!}}{z} \right)
\right|.
\]

It follows from~\eqref{size3} and~\eqref{big} that, if $N$ is
sufficiently large, then
\begin{eqnarray*}
|f(z)| & \geq & \frac{1}{4}  |z|^{m+1} \left(
\prod_{n=1}^{N-2}\prod_{i =
(n!)^2}^{((n+1)!)^2-1}\frac{1}{s_i^{n!}}\right) \prod_{i =
((N-1)!)^2}^{m-1}
\frac{1}{s_i^{(N-1)!}}\\
& \geq & \frac{1}{4} |z|^{m+1}\frac{1}{s_{m-1}^{2(N-1)!}} \\
& \geq & \frac{1}{4} |z|^{m+1 - 2(N-1)!/m} = \frac{1}{4} s_m^{k(m+1 - 2(N-1)!/m)}\\
& = & \frac{1}{4} s_m^{k(m+1)(1-2(N-1)!/m(m+1))} = \frac{1}{4} s_{m+1}^{k(1-2(N-1)!/m(m+1))}\\
& \geq & s_{m+1}^{k(1-2/m^{3/2}))}\geq s_{m+1}^2.
\end{eqnarray*}
Also, if $N$ is sufficiently large, then
\begin{eqnarray*}
|f(z)| & \leq &  |z|^{m+2} \prod_{n=1}^{N-2}\prod_{i =
(n!)^2}^{((n+1)!)^2-1}\frac{1}{s_i^{n!}} \prod_{i =
((N-1)!)^2}^{m-1}
\frac{1}{s_i^{(N-1)!}}\\
& \leq &  |z|^{m+2} = s_m^{k(m+2)}\\
& \leq &  s_m^{k(m+2)} = s_{m+1}^{k(1+1/(m+1))}\\
& \leq & s_{m+1}^{k+1} \leq s_{m+1}^{N! - 2},
\end{eqnarray*}
since $k \leq N! - 3 $. As $m = (N!)^2 - 1$, this shows that $f(z) \in A_{m+1,N}$.

Since
\[
 \{z: |z| = s_m^{(N-1)!+1}\} \subset \{z: \left|1- z/s_m^{(N-1)!}\right| > 1/2\},
\]
it also follows from the last set of inequalities that, if $N$ is sufficiently large, then
\[
M(s_m^{(N-1)!},f) < M(s_m^{(N-1)!+1},f) \leq s_{m+1}^{(N-1)^2+2}.
\]
This completes the proof of Example~\ref{exa2}.
\end{proof}

In Theorem~\ref{main2c} we showed that, if $f$ has a multiply connected {\wand}~$U$ and $B_n \subset f^n(U)$ is the annulus defined in~\eqref{Bndef}, for large $n \in \N$, then there is a close relationship between the inner and outer boundary components of $f^n(U)$ and the inner and outer boundary components of $B_n$. Our final example shows that it is not possible to strengthen the conclusions of Theorem~\ref{main2c} part (b)(ii) to give $\lim_{n \to \infty} \overline{b}_n/b_n = 1$.

For Example~\ref{exa3} we have $\lim_{n \to \infty} \overline{b}_n/b_n = \infty$; that is, the ratio of the logarithm of the modulus of the largest point on the outer boundary of $f^n(U)$ to the logarithm of the modulus of points on the outer boundary component of the annulus $B_n$ tends to infinity as $n \to \infty$. This contrasts with the fact that this ratio tends to~$1$ when the modulus of the largest point on the outer boundary is replaced by the modulus of the {\it smallest} point on the outer boundary. It also shows that the outer boundary component can be much more distorted than the inner boundary component for which both of these ratios tend to~$1$.

Example~\ref{exa3} also gives an alternative way of constructing a {\tef} with a multiply connected {\wand} $U$ for which the limits $a$ and $b$ defined in Theorem~\ref{main2c} satisfy $b/a < \infty$. We note that in this example the sequence $(r_n)$ plays a different role from that in the earlier sections of the paper.

\begin{example}\label{exa3}
 Let $c>1$. Then there exists an entire function $f$ and sequences $(r_n)$ and $(R_n)$ tending to $\infty$ such that $r_n<R_n$,
\begin{equation}\label{a} f(A(r_n,R_n))\subset A(r_{n+1},R_{n+1})
\end{equation}
and
\begin{equation}\label{b} f([R_n,R_n^n])\subset A(r_{n+1},R_{n+1})
\end{equation}
for all $n\in\N$. Moreover, if $U_n$ is the component of $F(f)$ containing $A(r_n,R_n)$ and if $A(t_n,T_n)$ is the maximal annulus with
\[A(r_n,R_n)\subset A(t_n,T_n)\subset U_n\;,\]
then
\begin{equation}\label{c} T_n< t_n^c\ .
\end{equation}
\end{example}
{\it Remark.} It follows from Theorem~\ref{main2a} that $A(t_n,T_n) = B_n$, for large $n \in \N$, where $B_n$ is the annulus defined in~\eqref{Bndef}. Hence, in the notation of Theorem~\ref{main2c}, $b/a = \lim_{n \to \infty} \log T_n/ \log t_n \leq c$ and for large $n \in \N$ we have $[R_n,R_n^n] \subset U_n$, so
\[
 \frac{\overline{b}_n}{b_n} \geq \frac{\log R_n^n }{\log T_n} > \frac{\log R_n^n}{\log t_n^c} > \frac{\log R_n^n}{\log R_n^c} = \frac nc.
\]
Thus $\lim_{n \to \infty} \overline{b}_n/ b_n = \infty$ as claimed.
 \begin{proof}[Proof of Example~\ref{exa3}]
 We define a sequence $(m_n)$ of positive integers and a sequence $(S_n)$ of real numbers greater than $e$ by recursion. First we choose $m_1$ and $S_1$ large. Assuming that $n\in \N$ and that $m_n$ and $S_n$ have been specified, we put
\begin{equation}\label{d}
\delta_n=\frac{1}{\sqrt{\log S_n}},
\end{equation}
\begin{equation}\label{e}
\displaystyle S_{n+1}=\displaystyle S_n^{(1-\delta_n)m_n+n^3}
\end{equation}
and we choose
\begin{equation}\label{f}
m_{n+1}\geq   S_{n+1}^{m_n+1}\,.
\end{equation}
It is easy to see that the sequences $(m_n)$ and $(S_n)$ tend to $\infty$, and in fact they do so very rapidly.

With
\begin{equation}\label{g}
a_n=S_n^{(1-\delta_n)m_n}=S_{n+1}S^{-n^3}_n,
\end{equation}
we put
\[g_n(z)=a_n\left(\exp\left(-\left(z/S_n\right)^{m_n}\right)-1\right)\]
and
\[g(z)=\sum^\infty_{n=1}g_n(z)\,.\]
We will see below that this series converges locally uniformly in $\C$
so that $g$ is entire.

We now choose $\alpha\in \left(1/c,1\right)$ and put
\[h(z)=\prod^\infty_{n=1}\left(1-\frac{z}{S_n^\alpha}\right)\,.\]
Since $(S_n)$ tends to infinity rapidly, this infinite product converges locally uniformly in $\C$.

The function $f$ is now defined by
\[f(z)=g(z)h(z)\] and the sequence $(R_n)$ is defined by
\begin{equation}\label{h}
\left(\frac{R_n}{S_n}\right)^{m_n}=\frac{1}{2}\,.
\end{equation}

In order to define the sequence $(r_n)$ we first define a sequence $(\beta_n)$ of real numbers recursively by
\begin{equation}\label{i}
\beta_{n+1}=\frac{\beta_n-3\delta_n}{1-\delta_n}\ ,
\end{equation}
with $\beta_1\in (\alpha,1)$. Since $(\delta_n)$ tends to zero rapidly we see that if $S_1$ and $m_1$ are chosen large enough, then the sequence $(\beta_n)$ converges to a limit $\beta\in (\alpha,1)$. Clearly we have $\beta<\beta_{n+1}<\beta_n<1$ for all $n\in \N$. The sequence $(r_n)$ is now defined by
\[r_n=R^{\beta_n}_n\,.\]
Before we prove that the sequences $(r_n)$ and $(R_n)$ have the desired properties, we collect some elementary estimates relating them to the other sequences
defined. First we note that
\begin{equation}\label{j}
\tfrac{1}{2}S_n\leq R_n<S_n
\end{equation}
by \eqref{h}, and thus
\begin{equation}\label{k}
S_{n+1}\geq R_{n+1}\geq r_{n+1}\geq
R_{n+1}^\beta\geq \frac{1}{2^\beta}S_{n+1}^\beta
\geq S_n^{\beta(1-\delta_n)m_n}\geq 4e^{m_n},
\end{equation}
by \eqref{e} for large $n$. Using \eqref{d} and \eqref{f} we see that
\begin{equation}\label{k1}
 m_n\geq m_n\delta_n\geq n^3,
\end{equation}
for large $n$. By \eqref{j} and \eqref{k1},
\[
r_{n+1}\leq
S_{n+1}^{\beta_{n+1}}
= S_n^{\beta_{n+1}(1-\delta_n)m_n+\beta_{n+1} n^3}
=S_n^{\beta_{n}m_n-3\delta_n m_n+\beta_{n+1} n^3}
\leq S_n^{m_n-2\delta_n m_n}
\leq \tfrac{1}{8}S_n^{(1-\delta_n)m_n}
\]
and thus, by~\eqref{g},
\begin{equation}\label{l}
r_{n+1}\leq \tfrac{1}{8}a_n,
\end{equation}
for large $n$.
Finally,
\begin{equation}\label{m}
n\leq a_n\leq \tfrac{1}{4}S_{n+1}\leq \tfrac{1}{2}R_{n+1},
\end{equation}
for large $n$, by \eqref{g} and \eqref{j}.

We shall repeatedly use the inequalities
\begin{equation}\label{m1}
\tfrac{1}{2}|w|\leq \left|e^w-1\right|\leq 2|w|\;\;\mbox{for } |w|\leq\tfrac{1}{2}\,.
\end{equation}
In order to prove \eqref{a} we note first that \eqref{h} and \eqref{m1} yield
\begin{equation}\label{n}
|g_n(z)|\leq 2 a_n \left(\frac{R_n}{S_n}\right)^{m_n}=a_n,\;\;\mbox{for }|z|\leq R_n.
\end{equation}
It also follows from \eqref{d}, \eqref{g} and \eqref{m1} that if $|z|\leq R_n$, then
\begin{eqnarray*}
\sum^\infty_{j=n+1}|g_j(z)| & \leq & 2\sum^\infty_{j=n+1}a_j\left(\frac{R_n}{S_j}\right)^{m_j}\\
& = & 2\sum^\infty_{j=n+1}S_j^{-\delta_jm_j}R^{m_j}_n\\
& = & 2\sum^\infty_{j=n+1}\exp\left(m_j\left(\log R_n-\sqrt{\log S_j}\right)\right)\,.
\end{eqnarray*}

For large $n$ we have $\sqrt{\log S_{n+1}}\geq \log R_n+1$, by \eqref{e}, \eqref{f} and \eqref{j}, and this implies that
\begin{equation}\label{o} \sum^\infty_{j=n+1}|g_j(z)|\leq 1\leq \tfrac{1}{4}r_{n+1},\;\;\mbox{for }|z|\leq R_n,
\end{equation}
if $n$ is large. This also shows that the series for $g$ converges locally uniformly in $\C$ and thus defines an entire function.

For $|z|\leq R_n$ we also have
\begin{eqnarray*}
\sum^{n-1}_{j=1}|g_j(z)| &\leq & \sum^{n-1}_{j=1}a_j\left(\exp\left(\left(R_n/S_j\right)^{m_j}\right)+1\right)\\
& \leq & 2(n-1)a_{n-1}\exp\left(R_n^{m_{n-1}}\right)\\
& \leq & \tfrac{1}{2}R_n^2\exp\left(R_n^{m_{n-1}}\right),
\end{eqnarray*}
by \eqref{m}. Noting that $\frac{1}{2}x^2\leq e^x$ for $x\geq 0$ we obtain
\begin{equation}\label{p} \sum^{n-1}_{j=1}|g_j(z)|\leq \exp\left(R_n^{m_{n-1}+1}\right)\leq
e^{m_n}\leq \tfrac{1}{4}r_{n+1},\;\;\mbox{for } |z|\leq R_n,
\end{equation}
by \eqref{f}, \eqref{j} and \eqref{k}. Combining \eqref{n}, \eqref{o} and \eqref{p} we obtain
\[|g(z)|\leq a_n+\tfrac{1}{2}r_{n+1}\] and thus
\begin{equation}\label{q} |g(z)|\leq 2a_n,\;\;\mbox{for } |z|\leq R_n,
\end{equation}
by \eqref{l}.

For $|z|\leq R_n$ we also have
\[\left|\prod^n_{j=1}\left(1-\frac{z}{S_j^\alpha}\right)\right|\leq \prod^n_{j=1}\left(1+\frac{R_n}{S_j^\alpha}\right)\leq R^n_n \]
and
\[\left|\prod^\infty_{j=n+1}\left(1-\frac{z}{S_j^\alpha}\right)\right|\leq \prod^\infty_{j=n+1}\left(1+\frac{R_n}{S_j^\alpha}\right)\leq 2\]
if $n$ is large, so
\begin{equation}\label{r} |h(z)|\leq 2R^n_n,\;\;\mbox{for } |z|\leq R_n.
\end{equation}
Combining \eqref{q} and \eqref{r} with \eqref{e} and \eqref{g} we obtain, for $|z|\le R_n$,
\begin{equation}\label{s}
|f(z)|
\leq 4a_nR^n_n
\leq 4a_nS^n_n
= 4S_n^{(1-\delta_n)m_n+n}
\leq \tfrac{1}{2}S_{n+1}\leq R_{n+1}.
\end{equation}
For $r_n\leq|z|\leq R_n$ we deduce from \eqref{h} and \eqref{m1} that
\[
|g_n(z)| \geq  \frac{1}{2}a_n\left(\frac{r_n}{S_n}\right)^{m_n}
 = \frac{1}{2}S_n^{(1-\delta_n)m_n}\left(R_n^{m_n}\right)^{\beta_n}S_n^{-m_n}
 = \frac{1}{2^{1+\beta_n}}S_n^{m_n(\beta_n-\delta_n)}\,,
\]
and hence, using \eqref{i} and \eqref{k1}, that
\[
|g_n(z)| \geq  2S_n^{m_n(\beta_n-3\delta_n)+n^3}
          =    2S_n^{m_n(1-\delta_n)\beta_{n+1}+n^3}
          \geq 2S_{n+1}^{\beta_{n+1}}
          \geq 2R_{n+1}^{\beta_{n+1}}
          =    2r_{n+1}\,.
\]
Combining this with \eqref{o} and \eqref{p} yields
\begin{equation}\label{t} |g(z)|\geq r_{n+1},\;\;\mbox{for }r_n\leq |z|\leq R_n\,.
\end{equation}

Classical results say that entire functions of small growth are large outside small neighborhoods of the zeros. Applying such results, or using elementary estimates, it is not difficult to see that
\begin{equation}\label{u}
|h(z)|\geq 1,\;\;\mbox{for }r_n\leq |z|\leq R_n
\end{equation}
for large $n$.
Combining \eqref{t} and \eqref{u} we obtain
\[|f(z)|\geq r_{n+1},\;\;\mbox{for }r_n\leq |z|\leq R_n\]
if $n$ is large. This, together with \eqref{s}, proves \eqref{a} for large $n$.

In order to prove \eqref{b}, let $x\in [R_n,R_n^n]$. Then
\[|g_n(x)|=a_n\left(1-\exp\left(-\left(x/S_n\right)^{m_n}\right)\right)\]
and hence
\begin{equation}\label{v} \tfrac{1}{4}a_n\leq\left(1-e^{-1/2}\right)a_n\leq|g_n(x)|\leq a_n\,.
\end{equation}
We also have
\begin{equation}\label{w}
\sum^{n-1}_{j=1}|g_j(x)|\leq \sum^{n-1}_{j=1}a_j\leq(n-1)a_{n-1}\leq \tfrac{1}{4}R_n^2\leq \tfrac{1}{4}r_{n+1}
\end{equation}
by \eqref{m} and similarly as in \eqref{o} we obtain
\begin{equation}\label{x}
\sum^\infty_{j=n+1}|g_j(x)|\leq 1\leq\tfrac{1}{4}r_{n+1}\,.
\end{equation}
Combining \eqref{v}, \eqref{w} and \eqref{x} with \eqref{l} and \eqref{m} we obtain
\begin{equation}\label{y}
|g(x)|\leq 2 a_n
\end{equation}
and
\begin{equation}\label{z}
|g(x)|\geq \tfrac{1}{4}a_n-\tfrac{1}{2}r_{n+1}\geq r_{n+1}\,.
\end{equation}
Similarly, as in \eqref{r} and \eqref{u}, we have
\begin{equation}\label{z1}
1\leq |h(x)|\leq 2 |x|^n\leq 2R_n^{n^2}\,.
\end{equation}

We deduce from \eqref{y}, \eqref{z} and \eqref{z1} that
\[|f(x)|\leq 4a_nR^{n^2}_n\leq 4S_n^{(1-\delta_n)m_n+n^2}\leq \tfrac{1}{2}S_{n+1}\leq R_{n+1}\]
and
\[|f(x)|\geq r_{n+1}\,.\]
Thus \eqref{b} follows.

To prove \eqref{c} we note that  $t_n\geq S_n^\alpha$ since $f(S_n^\alpha)=0$. We shall prove that $T_n\leq(1+o(1))S_n$ by showing that there exists a zero $z_n$ of $f$ satisfying $|z_n|\sim S_n$ as $n\to \infty$. To this end we consider the curve
\[
\gamma_n=\left\{S_n\left(2\pi i+\tfrac{1}{2}e^{i\varphi}\right)^{1/m_n}:-\pi\leq\varphi\leq \pi\right\}\,,
\]
with the principal branch of the $m_n$-th root. For $z=S_n\left(2\pi i+\frac{1}{2}e^{i\varphi}\right)^{1/m_n} \in \gamma_n$
we have
\[
|g_n(z)|  =  a_n\left|\exp\left(-\left(z/S_n\right)^{m_n}\right)-1\right|
          =  a_n\left|\exp\left(-\tfrac{1}{2}e^{i\varphi}\right)-1\right|
          \geq  \tfrac{1}{4}a_n
          \geq  2r_{n+1}
\]
by \eqref{l} and \eqref{m1}. Similarly, as before, we see that
\[
\sum^{n-1}_{j=1}|g_j(z)|\leq \tfrac{1}{4}r_{n+1}\;\;\mbox{and}\;\; \sum^\infty_{j=n+1}|g_j(z)|\leq \tfrac{1}{4}r_{n+1}\,,
\]
for $z\in \gamma_n$. Thus
\[|g(z)-g_n(z)|<|g_n(z)|,\;\;\mbox{for } z\in \gamma\,.\]
Since $g_n$ has the zero $S_n(2\pi i)^{1/m_n}$ in the interior of $\gamma_n$, Rouch\'e's theorem implies that $g$ has a zero $z_n$ in the interior of $\gamma_n$. Clearly $|z_n|\sim S_n$, so it follows that $T_n\leq(1+o(1))S_n\leq t_n^{1/\alpha + o(1)}$. Since $\alpha\in (1/c,1)$ we conclude that \eqref{c} holds for large~$n$.

Thus we have proved that \eqref{a},  \eqref{b} and \eqref{c} hold
for large~$n$. Actually, by choosing $m_1$ and $S_1$ large enough,
we may assume that these estimates hold for all~$n$. This completes the proof of Example~\ref{exa3}.
\end{proof}

{\it Acknowledgements}\;We thank N\'uria Fagella for a useful conversation that led us to include Theorem~\ref{curves} and Jian-Hua Zheng for drawing our attention to the question answered by Theorem~\ref{stable}. The last two authors thank Walter Bergweiler and the University of Kiel for their hospitality during a visit in 2009, when this work started.

\end{document}